\documentclass{article}

\author{Ofir Gorodetsky, Brad Rodgers}

\newcommand{\Addresses}{{
  \bigskip
  \footnotesize

  \textsc{Mathematical Institute, Woodstock Road, Oxford OX2 6GG, UK}\par\nopagebreak
  \textit{E-mail address:} \texttt{ofir.goro@gmail.com}

  \medskip

  \textsc{Department of Mathematics and Statistics, Queen's University, Kingston, Ontario, K7L 3N6, Canada}\par\nopagebreak
  \textit{E-mail address:} \texttt{brad.rodgers@queensu.ca}

}}

\title{The variance of the number of sums of two squares in $\FF_q[T]$ in short intervals}
\date{}

\usepackage[margin=1.5in]{geometry}
\usepackage[all]{xy}
\usepackage{amssymb}
\usepackage{amsfonts}
\usepackage{amsthm}
\usepackage{amsmath}
\usepackage{hyperref}
\usepackage{mathtools}
\mathtoolsset{showonlyrefs}
\usepackage{url}
\usepackage{enumerate}
\usepackage{color}
\usepackage{subcaption}
\usepackage{graphicx}

\newtheorem*{thm*}{Theorem}
\newtheorem{thm}{Theorem}[section]

\newtheorem{lem}[thm]{Lemma}  
\newtheorem{proposition}[thm]{Proposition}
\newtheorem{cor}[thm]{Corollary}

\newtheorem{conj}{Conjecture} 

\theoremstyle{definition}

\theoremstyle{remark}
\newtheorem{remark}{Remark}

\newcommand{\FF}{\mathbb{F}}
\newcommand{\Var}{\mathrm{Var}}
\newcommand{\EE}{\mathbb{E}}

\numberwithin{equation}{section}

\begin{document}

\maketitle

\begin{abstract}
Consider the number of integers in a short interval that can be represented as a sum of two squares. What is an estimate for the variance of these counts over random short intervals? We resolve a function field variant of this problem in the large $q$ limit, finding a connection to the $z$-measures first investigated in the context of harmonic analysis on the infinite symmetric group. A similar connection to $z$-measures is established for sums over short intervals of the divisor functions $d_z(n)$. We use these results to make conjectures in the setting of the integers which match very well with numerically produced data. Our proofs depend on equidistribution results of N. Katz and W. Sawin.

\end{abstract}

\section{Introduction}\label{sec:funcs}
\subsection{Classical theory}\label{subsec:classical}
Consider the set $S = \{ n^2 +m^2 : n,m \in \mathbb{Z}\}$ of integers representable as sums of two perfect squares, and let $b \colon \mathbb{Z}_{>0} \to \mathbb{C}$ be the indicator function of the set $S$. Landau \cite{land} first gave an estimate for the number of positive integers no more than $x$ that belong to $S$:
\begin{equation}
\label{Landau_asymp}
B(x):= \sum_{1 \leq n \leq x} b(n) = K\frac{x}{\sqrt{\log x}} + O\Big( \frac{x}{(\log x)^{3/2}}\Big),
\end{equation}
where
\begin{equation}
\label{K_def}
K := \frac{1}{\sqrt{2}} \prod_{p \equiv 3 \bmod 4}(1-p^{-2})^{-1/2} \approx 0.764
\end{equation}
is the Landau-Ramanujan constant. Thus, roughly stated, the likelihood that a random integer near $X$ will be the sum of two squares is around $K/\sqrt{\log X}$. In fact there exists a more accurate approximation than \eqref{Landau_asymp} for the sum on the left hand side, and while we will come to this later we do not need this more precise information just yet.

In this paper we are motivated by the goal of understanding how many elements of $S$ lie in a random short interval $[n,n+H]$, where $n$ is chosen randomly from $[1,X]$ and $H = o(X)$. For $H$ proportional to $\sqrt{\log X}$, one expects the count to be distributed like a Poisson random variable:
\begin{conj}
\label{poisson}
For $\lambda$ a fixed parameter, let $H:= \lambda \sqrt{\log X}/K$. Then
\begin{equation}
\label{poisson_moments}
\frac{1}{X} \sum_{1 \leq n \leq X}\Big( \sum_{m \in [n,n+H]} b(m) \Big)^k \sim m_k(\lambda),
\end{equation}
where $m_k(\lambda)$ is the $k$-th moment of a Poisson distributed random variable with parameter $\lambda$.
\end{conj}
Note that for $k=1$, we have $m_1(\lambda) = \lambda$. The $k=1$ case of \eqref{poisson_moments} is just the statement that on average $\lambda$ elements of $S$ lie in an interval of this size, and it is easy to see that this in fact follows from Landau's result \eqref{Landau_asymp}. For $k=2$, we have $m_2(\lambda) = \lambda$ also, and in this case \eqref{poisson_moments} was shown by Smilansky \cite{smilansky2013} to follow conditionally on Hardy-Littlewood type conjectures for the function $b(n)$. Moments with $k\geq 3$ were recently studied by Freiberg, Kurlberg and Rosenzweig \cite{freiberg2017} who showed conditioned on Hardy-Littlewood type conjectures that \eqref{poisson_moments} is true for all $k$. (This work is analogous to work of Gallagher \cite{gallagher1976distribution}, who showed conditionally on Hardy-Littlewood conjectures that the number of primes in short intervals of this sort also is distributed like a Poisson random variable.)

\begin{figure}[h]
  \includegraphics[width=\linewidth]{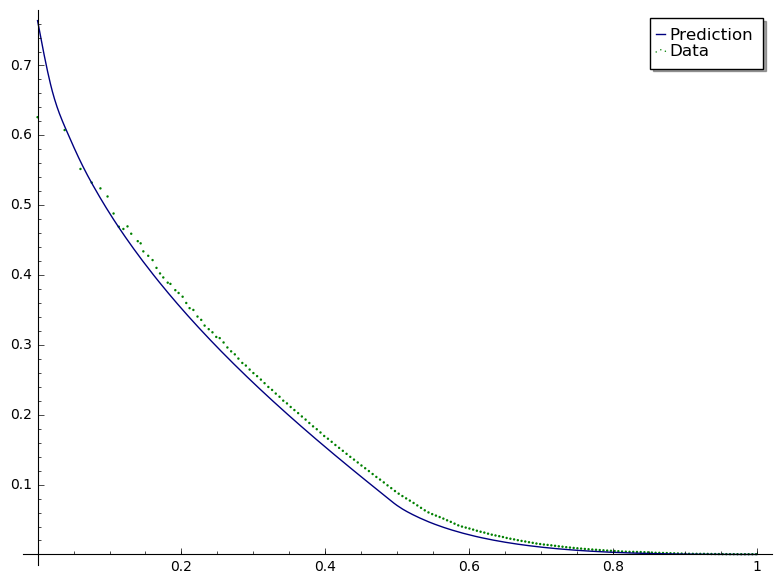}
  \caption{Numerically produced data compared to the $z$-measure induced prediction given in Conjecture \ref{b_variance_conj} for variance in short intervals. Let $V_b(X;H)$ be the variance of counts of $S$ in random short intervals $[x,x+H]$ for $x \leq X$. For $X = 10^8$ and $H\leq X$, set $\delta = \delta_H = \log(H)/\log(X)$. For a selection of $H$, we plot the points $(\delta, V_b(X,H)/(H/\sqrt{\log X}))$ under the label \textbf{data}, and the curve $(\delta, K\, G(1-\delta))$ under \textbf{prediction}. See Section \ref{sec:making_integer_conjectures} for further discussion.}
\label{fig:data_and_prediction_short_interval}
\end{figure}

Conjecture \ref{poisson} might lead one at first to believe more generally that seen at a large enough scale, the elements of $S$ that lie in a random short interval should resemble a Poisson point process, meaning that they should be laid down more or less at random, and the variance of their count should end up being roughly equal to the number expected to lie in the interval. In this paper we suggest that this pattern breaks down for short intervals larger than those considered in Conjecture \ref{poisson}; in particular we suggest for short intervals of size $X^\delta$ that the variance will be smaller than a straightforward extension of this conjecture would suggest. (There is an analogy to what happens for primes; see \cite{soundararajan2007distribution}.)
In fact we find a connection to the $z$-measures that have been investigated in the context of harmonic analysis on the symmetric group (see for instance \cite{kerov1993harmonic,olshanski2003introduction,borodin1998point}). It is common to heuristically justify deviations from the Poisson setting by making use of Hardy-Littlewood type conjectures, however we do not justify the presence of $z$-measures this way.

Instead, our evidence for their appearance is a theorem that we prove for a function field analogue of this problem. To state our result requires a little background that we give below; our main result is Theorem \ref{main_thm_limit}, and we make a conjecture for the integers for random short intervals in Conjecture \ref{b_variance_conj} and a similar conjecture for the integers for random sparse arithmetic progressions in Conjecture \ref{b_variance_conj_q}. These conjectures match up very well with numerical data (see Figures \ref{fig:data_and_prediction_short_interval} and \ref{fig:data_and_prediction_AP}).

\subsection{The function field analogy}\label{subsec:ffanalogue}

We let $q$ be an odd prime power and let $\FF_q[T]$ be the polynomial ring over the finite field $\FF_q$ with $q$ elements. Let $\mathcal{M}_{n,q}$ denote the set of monic polynomials of degree $n$ in $\FF_q[T]$, and let $\mathcal{M}_q = \cup_{n \ge 0} \mathcal{M}_{n,q}$ denote the set of all monic polynomials in $\FF_q[T]$. By a well-known analogy $\mathcal{M}_q$ serves as a substitute for the set of positive integers. In \cite{barysoroker2016on}, Bary-Soroker, Smilansky and Wolf studied an analogue of Landau's problem in $\FF_q[T]$ by introducing the following set and indicator function, which by abuse of notation we will also denote by $S$ and $b$: 
\begin{equation*}
S := \{ A^2 + TB^2:\; A,B \in \mathcal{M}_q\}
\end{equation*}
\begin{equation}\label{defb}
\begin{split}
& b\colon \mathcal{M}_q \to \mathbb{C}, \\
b(f) &= \begin{cases} 1 & \text{if } f \in S \\ 0 & \text{otherwise} \end{cases}.
\end{split}
\end{equation}
The analogy with integers can be seen in the following way: a positive integer lies in $S$ if an only if it is the norm of some element of $\mathbb{Z}[i]$, and an element of $\mathcal{M}_q$ lies in $S$ if and only if it is the norm of some element of $\FF_q[\sqrt{-T}]$. (One may also work with $A^2-\alpha B^2$ for some $\alpha \in \FF_q^{\times}\setminus (\FF_q^{\times})^2$, see \cite{leahey} and \cite[Sec.~5]{gorodetsky2017}.) 

In $\FF_q[T]$, the mean value of $b$ can be estimated as follows \cite[Thm.~1.1]{gorodetsky2017}:
\begin{equation}\label{Gorodetsky_asymp}
q^{-n}\sum_{f \in \mathcal{M}_{n,q}} b(f) = K_q  \binom{n-\frac{1}{2}}{n}  \left( 1 + O\left(\frac{1}{qn}\right) \right),
\end{equation}
where the implied constant is absolute, and the constant $K_q$ is positive and is an analytic function of $q^{-1}$. The constant $K_q$ is given by
\begin{equation*}
K_q = (1-q^{-1})^{-1/2} \prod_{\chi_2(P)=-1}(1-q^{-2\deg P})^{-1/2} = 1+O\left( \frac{1}{q} \right),
\end{equation*}
where $\chi_2$ is the unique non-principal quadratic Dirichlet character modulo $T$. By Stirling's formula, $\binom{n-1/2}{n} = 1/\sqrt{\pi n} + O(1/n^{3/2})$, and so \eqref{Gorodetsky_asymp} has a resemblance to \eqref{Landau_asymp}.

We take a notion of short intervals in $\FF_q[T]$ used prominently in \cite{keating2014variance}. Let $n$ be a positive integer, let $0 \le h \le n-1$, and let $A \in \mathcal{M}_{n,q}$. A short interval around $A$ of size $q^{h+1}$ is the subset 
\begin{equation*}
I(A;h) :=\{ A + g : \deg g \le h\} \subseteq \mathcal{M}_{n,q}.
\end{equation*}
Given an arithmetic function $\alpha\colon \mathcal{M}_{n,q} \to \mathbb{C}$, we let 
\begin{equation*}
\nu_{\alpha}(A;h) := \sum_{f \in I(A;h) } \alpha(f)
\end{equation*}
be the sum of $\alpha$ over the short interval $I(A;h)$. Note that $\nu_b(A;h)$ counts the number of elements of $S$ that lie in the short interval $I(A;h)$. We are interested in the variance of $\nu_{\alpha}(A; h)$ as $A$ varies, where we make the definition
\begin{equation}\label{des}
\Var_{A \in \mathcal{M}_{n,q}}\Big(\nu_{\alpha}(A; h)\Big) := \frac{1}{q^{n}} \sum_{A \in \mathcal{M}_{n,q}} \left| \nu_{\alpha}(A ;h) - \EE_{B \in \mathcal{M}_{n,q}}\big( \nu_\alpha(B;h)\big) \right|^2,
\end{equation}
where $\EE_{B \in \mathcal{M}_{n,q}} \big(\nu_\alpha(B;h)\big)$ is the mean value of $\nu_{\alpha}(B; h)$ which equals 
\begin{equation*}
\EE_{B \in \mathcal{M}_{n,q}} \big(\nu_\alpha(B;h)\big) := \frac{1}{q^n} \sum_{B \in \mathcal{M}_{n,q}} \nu_\alpha(B;h) = q^{h+1}\cdot \frac{1}{q^n} \sum_{f \in \mathcal{M}_{n,q}} \alpha(f).
\end{equation*}

Our main result is an evaluation of the variance of $\nu_b(A;h)$ in a large $q$ limit. The evaluation involves the $z$-measure on partitions introduced in \cite{kerov1993harmonic}. The $z$-measures arise in an evaluation of certain integrals over the unitary group (Theorem \ref{RMT_to_zmeasures}), which may be of independent interest. 

We give a brief introduction to these measures in Section \ref{sec:zmeasures}, but for the moment we discuss only the notation;  recall that we write $\lambda \vdash n$ to indicate that $\lambda$ is a partition of $n$ and $\lambda_1$ to indicate the largest part of a partition $\lambda$. For parameters $z \in \mathbb{C}$ and $n\in \mathbb{N}$, the $z$-measure is a probability measure $M_z^{(n)}(\lambda)$ on the set of partitions $\lambda \vdash n$. In fact these $z$-measures are a generalization of the well-known Plancherel measure on partitions. The notation $\mathbb{P}_z^{(n)}(\lambda_1 \leq N)$ denotes the obvious thing, namely
\begin{equation*}
\mathbb{P}_z^{(n)}(\lambda_1 \leq N):=\sum_{\substack{\lambda\vdash n \\ \lambda_1 \leq N}} M_z^{(n)}(\lambda).
\end{equation*}
The actual definition of these probability measures $M_z^{(n)}$ will be given in Section \ref{sec:zmeasures}. (By convention we set $M_z^{(0)}(\lambda_1 \leq N) = 1$ for any $N$.)  We show
\begin{thm}
\label{main_thm_Variance}
For a fixed odd prime $p$, and fixed $n \geq 6$, take $0 \leq h \leq n-7$ and let $N:= n-h-1$ and $q = p^k$. Define
\begin{equation*}
T(n;N):= \sum_{j=0}^n \frac{(1/4)_j (1/4)_{n-j}}{j! (n-j)!} \mathbb{P}_{1/2}^{(j)}(\lambda_1 \leq N-1) \mathbb{P}_{1/2}^{(n-j)}(\lambda_1 \leq N).
\end{equation*}
For $N(N-1) \geq n$,
\begin{equation}
\label{Variance_to_T}
\Var_{A \in \mathcal{M}_{n,q}}\Big( \nu_b(A;h)\Big) = q^{h+1} T(n;N) + o(q^{h+1}),
\end{equation}
as $q \rightarrow \infty$ (that is $k\rightarrow \infty$).
\end{thm}
Here and throughout this paper $(x)_j := x(x+1)\cdots (x+j-1)$ is the \emph{rising factorial} Pochhammer symbol. 

\begin{remark} 
Likely this result is true for finite field sizes $q$ tending to infinity in an arbitrary fashion, but we are not able to prove it in this more general case. This owes to a crucial use in our proof of a recent theorem of Sawin \cite{sawin2018equidistribution}, which also requires this restriction.
\end{remark}

\begin{remark}
Theorem \ref{main_thm_Variance} complements work in \cite{bank2018sums}, which shows that in \emph{all} short intervals with $h\geq 2$, the count $\nu_b(A;h)$ is asymptotic to the mean as $q\rightarrow\infty$. 
\end{remark}

\begin{remark}
The reader should think of the sum defining $T(n;N)$ as being akin to a Riemann sum, so no single term predominates.
\end{remark}

\begin{remark}
Informally, one may think of the convolution structure in the sum defining $T(n;N)$ as arising because the arithmetic function $b$ can essentially be written as a convolution of the functions $d_{1/2}$ and $\chi_2 d_{1/2}$, where as we will discuss below, $\chi_2$ is a certain Dirichlet character and $d_{1/2}$ is a divisor function.
\end{remark}

We use this theorem to inform analogous conjectures in the setting of the integers in Section \ref{sec:making_integer_conjectures}. We require for this purpose an understanding of the limiting behavior of the expression on the right hand side as $h,n\rightarrow\infty$ with $h/n \rightarrow\delta\in (0,1)$. Note that if $h$ and $n$ are both sufficiently large and $h \sim \delta n$, then $N(N-1) \geq n$ and $N \geq 6$ will both be satisfied. The expression on the right hand side of \eqref{Variance_to_T} ends up being of order $O(q^{h+1}/\sqrt{n})$. In fact we can characterize the limiting behavior more exactly.

\begin{proposition}
\label{T_limit}
For $n, N \rightarrow \infty$ with $N/n\rightarrow s \in [0,1]$, we have
\begin{equation*}
T(n;N) = \frac{1}{\sqrt{\pi n}} G(s) + o\Big(\frac{1}{\sqrt{n}}\Big),
\end{equation*}
where for real $s$ we define
\begin{equation}
\label{G_def}
G(s) := \mathbb{P}\Big(1 - \frac{s}{\alpha_1} \leq Y \leq \frac{s}{\alpha_1'}\Big),
\end{equation}
for $Y$, $\alpha_1, \alpha_1'$ independent random variables, with $Y$ distributed as $\mathrm{Beta}(1/4,1/4)$ and $\alpha_1, \alpha_1'$ identically distributed copies of the largest part of the Thoma simplex distributed according to the spectral $z$-measure with parameters $1/2, 1/2$. (The spectral $z$-measure is defined in Section \ref{sec:zmeasures}.)
\end{proposition}

Note that the random variable $Y \sim \mathrm{Beta}(1/4,1/4)$ is defined by $\mathbb{P}(a \leq Y \leq b) := \sqrt{\pi}\Gamma(1/4)^{-2} \int_a^b t^{-3/4}(1-t)^{-3/4}\, dt$ for $a,b \in [0,1]$, with $Y \in [0,1]$ almost surely.

The random variables $\alpha_1$ and $\alpha_1'$ also lie in $[0,1]$ almost surely, but an explicit characterization of their distribution takes more space to give. Historically they arose in formulas for the characters of certain important representations in the infinite symmetric group (see \cite{kerov1993harmonic}), but more concretely they are the limiting distribution of the random variable $\lambda_1/n$ for $\lambda \vdash n$ drawn according to the $z$-measure of Theorem \ref{main_thm_Variance}. That such a limiting distribution even exists is not obvious, but was shown in \cite{olshanski1998point}. We discuss $z$-measures on the Thoma simplex in more detail in Section \ref{sec:zdefinitions}.

Plainly for all $s \in [0,1]$ we have $0 \leq G(s) \leq 1$. It also is easy to see (i) that $G(s)$ is non-decreasing (from the definition) and (ii) that $G(1) =1$ (from the fact that $Y, \alpha_1, \alpha_1' \in [0,1]$ almost surely). As a corollary of very recent work on $z$-measures of Korotkikh \cite{korotkikh2020transition} and Olshanski \cite{olshanski2018topological} we also have,

\begin{thm}
\label{G_positivity}
$G(s) > 0$ for all positive $s$.
\end{thm}

We prove Theorem \ref{G_positivity} in Appendix \ref{sec:Further}.

Using Theorem \ref{main_thm_Variance} and Proposition \ref{T_limit} together, we can write somewhat more succinctly,

\begin{thm}
\label{main_thm_limit}
For a fixed odd prime $p$ let $q = p^k$. If $h,n \rightarrow\infty$ in such a way that $h/n \rightarrow \delta \in (0,1)$, then
\begin{equation}
\lim_{q\rightarrow\infty} \frac{1}{q^{h+1}} \Var_{A \in \mathcal{M}_{n,q}}\Big( \nu_b(A;h)\Big) = \frac{G(1-\delta) + o(1)}{\sqrt{\pi n}},
\end{equation}
where the function $G(s)$ is defined in Proposition \ref{T_limit}.
\end{thm}

\subsection{The divisor functions \texorpdfstring{$d_z(n)$}{dz(n)}}\label{subsec:divisorfuncs}
The result we prove here for the indicator function $b$ has a close relationship to a related result for the generalized divisor function $d_z$, especially for $z=\frac{1}{2}$. Recall that over the integers the function $d_z(n)$ is defined by the generating series
\begin{equation}
\label{divisor_generating}
\zeta(s)^z = \sum_n \frac{d_z(n)}{n^s}\quad \text{for}\; \Re(s) > 1.
\end{equation}
Here we let $\zeta(s)^z:=  \exp(z \log \zeta(s))$, where the logarithm is the principal branch such that $\log \zeta(s)$ is real for $s > 1$ real. For the sake of conceptual simplicity we will stick to the case that $z$ is a positive real number, though one could extend our results to a larger range of $z$. Using Euler products one sees that $d_z(n)$ is a multiplicative function and satisfies
\begin{equation*}
d_z(n) = \prod_{p^\ell || n} \binom{\ell+z-1}{\ell}.
\end{equation*}

For $f \in \mathcal{M}_q$ we define $d_z(f)$ similarly by
\begin{equation}
\label{d_z_def}
d_z(f) = \prod_{P^\ell || f} \binom{\ell+z-1}{\ell},
\end{equation}
where the product is over all $P^\ell$ dividing $f$ such that $P^{\ell+1}$ does not divide $f$, where $P$ is an irreducible monic polynomial and $\ell \geq 1$.

We show that the variance of short interval sums of the function $d_z$ is also related to the $z$-measures introduced in the last subsection.

\begin{thm}
\label{d_z_variance}
Fix $z > 0$. Take $0 \leq h \leq n-5$ and let $N:= n-h-1$. We have
\begin{equation}
\Var_{A \in \mathcal{M}_{n,q}}\Big( \nu_{d_z}(A;h)\Big) = q^{h+1} \frac{(z^2)_n}{n!} \mathbb{P}_z^{(n)}(\lambda_1 \leq N-1) + O_n(q^{h+1/2})
\end{equation}
as $q \to \infty$.
\end{thm}

\begin{remark}
This generalizes a result of Keating, Rodgers, Roditty-Gershon and Rudnick \cite[Thm.~1.2]{keating2018sums}, who consider the result for $z \in \mathbb{N}$ and are able to find a simpler expression for the right hand side in this case.
\end{remark}

\begin{remark}
In Theorem \ref{d_z_variance}, $q$ can grow to infinity in an arbitrary fashion; we do not require that $q = p^k$ for a fixed prime $p$ as we did in Theorem \ref{main_thm_Variance}.
\end{remark}

As before we can characterize the limiting behavior of the right hand size as $n$ and $N$ grow:

\begin{proposition}
\label{Thoma_limit}
Fix $z > 0$ with $z\neq 1$. For $n, N \rightarrow \infty$ with $N/n \rightarrow s \in (0,1)$, we have
$$
\frac{(z^2)_n}{n!} \mathbb{P}_z^{(n)}(\lambda_1 \leq N-1) = \frac{n^{z^2-1}}{\Gamma(z^2)} F_z(s) + o(n^{z^2-1}),
$$
where for real $s$ we define
\begin{equation}
\label{CDF_defn}
F_z(s) := \mathbb{P}(\alpha_1^{(z)} \leq s),
\end{equation}
for $\alpha_1^{(z)}$ the largest part of the Thoma simplex distributed according to the spectral $z$-measure with parameters $z,z$.
\end{proposition}

For all $s \in [0,1]$ obviously $0\leq F_z(s) \leq 1$, and moreover (i) $F_z(s)$ is non-decreasing and (ii) $F_z(1) = 1$. The positivity of the functions $F_z(s)$ splits into two cases:

\begin{thm}
\label{thm:F_z_positivity}
For $k\geq 2$ an integer, $F_k(s)$ vanishes for $s \in [0,1/k]$ and is positive for $s > 1/k$.

For $z > 0$ with $z$ not an integer, $F_z(s) > 0$ for all positive $s$.
\end{thm}

For $k \ge 2$, Theorem \ref{thm:F_z_positivity} is just a restatement of work in \cite{keating2018sums}. For non-integer $z$, Theorem \ref{thm:F_z_positivity} is a corollary of recent work in \cite{olshanski2018topological, korotkikh2020transition}. We discuss this in more detail in Appendix \ref{sec:Further}, along with more properties of the function $F_z(s)$.

Using Theorem \ref{d_z_variance} and Proposition \ref{Thoma_limit} together, we have

\begin{thm}
\label{d_z_limit}
Fix $z > 0$ with $z\neq 1$. If $h,n \rightarrow\infty$ in such a way that $h/n \rightarrow \delta \in (0,1)$, then
\begin{equation}
\lim_{q\rightarrow\infty} \frac{1}{q^{h+1}}\Var_{A \in \mathcal{M}_{n,q}}\Big( \nu_{d_z}(A;h)\Big) = \Big(\frac{F_z(1-\delta)}{\Gamma(z^2)}  + o(1)\Big) n^{z^2-1}.
\end{equation}
\end{thm}

For the function $d_z(n)$, we discuss conjectures for the integers suggested by Theorems \ref{d_z_variance} and \ref{d_z_limit} in Section \ref{sec:making_integer_conjectures2}.

\subsection{Acknowledgements}

For discussions related to this paper, the authors would like to thank Alexei Borodin, P\"ar Kurlberg, Grigori Olshanski, Zeev Rudnick, Sasha Sodin and Eugene Strahov. The approach in Section \ref{C_less_Z} was outlined to us by Borodin. We also thank Adar Kahana for valuable help in producing some of the numerical graphs in this work. Finally we thank the referee for a number of helpful comments and corrections. The first author was supported by the European Research Council under the European Union's Seventh Framework Programme (FP7/2007-2013) / ERC grant agreement n$^{\text{o}}$ 320755. The second author was partly supported by the US NSF grants DMS-1701577 and DMS-1854398 and an NSERC grant.

\section{Dirichlet characters}
\subsection{The strategy of the proof}\label{subsec:strategy}

In this section we recall and develop the basic machinery necessary to prove Theorem \ref{main_thm_Variance}. (Theorem \ref{d_z_variance} by contrast is somewhat easier and will be seen to follow from machinery that has been developed elsewhere.) Our basic strategy is this: we relate the variance being considered to averages of sums of arithmetic functions against characters lying in a certain family; we relate these sums to sums over the zeros of $L$-functions; and finally we make crucial use of a recent theorem of W. Sawin characterizing the limiting distribution of zeros for the family of $L$-functions we have made use of. This relates the desired variance to a random matrix integral -- to our knowledge not previously considered in the literature -- and using symmetric function theory we give a combinatorial evaluation of this integral, relating it to the aforementioned $z$-measures.
\subsection{From short interval variance to character sums}
Recall that a Dirichlet character modulo a polynomial $Q \in \FF_q[T] \setminus \{0\}$ is a function $\chi \colon\FF_q[T] \to \mathbb{C}$ satisfying the properties $\chi(fg)=\chi(f)\chi(g)$ for $f,g \in \FF_q[T]$ (i.e. $\chi$ is completely multiplicative), $\chi(f) \neq 0$ if and only if $\gcd(f,Q)=1$, and $\chi(f)=\chi(g)$ whenever $f\equiv g \bmod Q$. The unique Dirichlet character modulo $Q$ which assumes the value $1$ on every polynomial coprime to $Q$ is called the principal character modulo $Q$, and is usually denoted by $\chi_0$ when it is understood what is our $Q$. We say that a character $\chi$ is \emph{even} if $\chi(c) = 1$ for all nonzero $c \in \FF_q$ (see e.g. \cite[Sec. 3.2]{keating2014variance}).

\begin{lem}\label{lemshortvar}
Let $h,n$ be two integers satisfying $0 \le h \le n-2$. Given a Dirichlet character $\chi\colon \FF_q[T] \to \mathbb{C}$ ($q$ odd) and an arithmetic function $\alpha \colon \FF_q[T] \to \mathbb{C}$, define
\begin{equation*}
S(n,\alpha,\chi) = \sum_{f \in \mathcal{M}_{n,q}} \alpha(f) \chi(f).
\end{equation*}
We have
\begin{equation*}
\Var_{A \in \mathcal{M}_{n,q}} \nu_{b}(A; h) = \frac{\sum_{\substack{\chi \bmod T^{n-h}\\\chi_0 \neq  \chi \text{ even}}} \left|\sum_{m=0}^{n}  S(m,b,\chi) \right|^2 }{q^{2(n-h-1)}}.
\end{equation*}
\end{lem}
\begin{proof}
For any $k \ge 0$, we have $b(T^{2k})=1$ and $b(T^{2k+1})=1$ since $T^{2k} = (T^k)^2 + T \cdot 0^2$ and $T^{2k+1} = 0^2 + T \cdot (T^k)^2$. By a result of Keating and Rudnick \cite[Lem.~5.4]{keating2016squarefree}, we have
\begin{equation*}
\Var_{A \in \mathcal{M}_{n,q}} \nu_{\alpha}(A; h) = \frac{\sum_{\substack{\chi \bmod T^{n-h}\\\chi_0 \neq  \chi \text{ even}}} \left|\sum_{m=0}^{n} \alpha(T^{n-m}) S(m,\alpha,\chi) \right|^2 }{q^{2(n-h-1)}}
\end{equation*}
for any arithmetic function $\alpha\colon\FF_q[T] \to \mathbb{C}$ that satisfies three conditions:
\begin{enumerate}
\item $\alpha(cf)=\alpha(f)$ holds for all $c \in \FF_q^{\times}$ and $f \in \FF_q[T]$,
\item $\alpha(T^k f)=\alpha(T^k)\alpha(f)$ holds for all $f$ coprime to $T$,
\item $\alpha(T^{\deg(f)}f(\frac{1}{T})) = \alpha(f)$ for any $f \in \FF_q[T]$ coprime to $T$.
\end{enumerate}
If we show that $b$ satisfies these properties, we are done, as $b(T^k)=1$ for all $k \ge 0$. The function $b$ satisfies the first property once we extend the definition of $b$ to non-monics as follows:
\begin{equation}
b(c\cdot f) = b(f)
\end{equation}
for all $c \in \FF_q^{\times}$ and $f \in \mathcal{M}_q$. (This means that $b(f)$ is the indicator of polynomials $f$ whose ideal $(f)$ is a norm of an ideal in $\FF_q[\sqrt{-T}]$.) The second property for $\alpha=b$ was established in \cite[Prop. 2.4]{barysoroker2016on}. We finish by verifying the last property. Since $f \mapsto T^{\deg (f)}f(1/T)$ is an involution, it suffices to show that if $b(f)=0$ then $b(T^{\deg(f)}f(1/T))=0$. By \cite[Prop. 2.4]{barysoroker2016on}, $b(f) = 0$ if and only if $P^{2k+1} \mid f$, $P^{2k+2} \nmid f$ for some $k \ge 0$ and a monic irreducible polynomial $P$ such that $P(0) \in (\FF_q)^{\times} \setminus ((\FF_q)^{\times})^2$. If we factor such $f$ (which is coprime to $T$) as
\begin{equation}
f= c \prod_{i=1}^{r} P_i^{e_i}
\end{equation}
where $P_i$ are distinct monic irreducibles and $c \in \FF_q^{\times}$, then the factorization of $T^{\deg(f)}f(1/T)$ is given by
\begin{equation}
T^{\deg f}f(\frac{1}{T}) = f(0) \prod_{i=1}^{r} (T^{\deg (P_i)}P_i(\frac{1}{T})/P_i(0))^{e_i}. 
\end{equation}
In particular, we have $\widetilde{P}^{2k+1} \mid T^{\deg(f)}f(1/T)$, $\widetilde{P}^{2k+2} \nmid T^{\deg(f)}f(1/T)$ for the monic irreducible $\widetilde{P}(T) = T^{\deg(P)}P(1/T)/P(0)$, which satisfies $\widetilde{P}(0) = 1/P(0) \in (\FF_q)^{\times} \setminus ((\FF_q)^{\times})^2$. This concludes the proof.
\end{proof}

\subsection{A computation related to the Dirichlet series of \texorpdfstring{$b$}{b}}\label{telescoping_comp}
In what follows we use the notation 
\begin{equation}
[u^n]\Big(\sum a_k u^k\Big) := a_n
\end{equation}
for formal power series $\sum a_k u^k$. We recall that the $L$-function of a Dirichlet character $\chi\colon \FF_q[T] \to \mathbb{C}$ is the power series
\begin{equation}
L(u,\chi) = \prod_{P}(1-\chi(P)u^{\deg(P)})^{-1},
\end{equation}
where the product runs over all monic irreducible polynomials.
\begin{lem}\label{snbchivia}
Let $\chi \colon \FF_q[T] \to \mathbb{C}$ be a Dirichlet character ($q$ odd). Define
\begin{equation}
S(n,b,\chi) = \sum_{f \in \mathcal{M}_{n,q}} b(f) \chi(f).
\end{equation}
Then
\begin{multline}\label{crazyiden}
S(n,b,\chi) = [u^n] \Big( \sqrt{L(u,\chi)L(u,\chi \cdot \chi_2)}\prod_{i\ge 1} \left( \frac{L(u^{2^i},\chi^{2^i})}{L(u^{2^i},\chi^{2^i} \cdot \chi_2)} \right)^{2^{-i-1}} \\
 \cdot (1-\chi(T)u)^{-1/2}\prod_{i \ge 1} (1-\chi^{2^i}(T)u^{2^i})^{2^{-i-1}} \Big).
\end{multline}
(The roots in the right hand side of \eqref{crazyiden} are chosen so that the constant terms remain 1.)
\end{lem}

Recall that $\chi_2$ is the unique non-principal Dirichlet character modulo $T$.

\begin{proof}
The lemma is equivalent to the following identity:
\begin{equation}\label{eulerprodb}
\begin{split}
\sum_{f \in \mathcal{M}_{q}}  b(f) \chi(f) u^{\deg f} &= \sqrt{L(u,\chi)L(u,\chi \cdot \chi_2)}\prod_{i\ge 1} \left( \frac{L(u^{2^i},\chi^{2^i})}{L(u^{2^i},\chi^{2^i} \cdot \chi_2)} \right)^{2^{-i-1}} \\
&\qquad \cdot (1-\chi(T)u)^{-1/2} \prod_{i \ge 1} (1-\chi^{2^i}(T)u^{2^i})^{2^{-i-1}}.
\end{split}
\end{equation}
We verify \eqref{eulerprodb} by comparing  the Euler product of both sides. By \cite[Prop. 2.4]{barysoroker2016on}, the function $b$ is multiplicative (that is, $b(fg)=b(f)b(g)$ for coprime $f,g \in \mathcal{M}_q$), and moreover at prime powers we have
\begin{equation}\label{eq:bonpowers}
b(P^k) = \begin{cases} 1 & \text{if}\;2 \mid k \text{ or } \chi_2(P)\in \{0,1\} \\ 0 & \text{otherwise}\end{cases}.
\end{equation}
Since $b \cdot \chi$ is multiplicative, \eqref{eq:bonpowers} implies that the left hand side of \eqref{eulerprodb} factors as 
\begin{equation}\label{eulerprodeasy}
\prod_{P:\chi_2(P)=1} (1-\chi(P)u^{\deg P})^{-1} \prod_{Q:\chi_2(Q)=-1} (1-\chi(Q^2)u^{2\deg Q})^{-1} (1-\chi(T)u^{\deg T})^{-1},
\end{equation}
where $P,Q$ denote monic irreducible polynomials. We have
\begin{equation}\label{lfunceuler}
\begin{split}
L(u,\chi) &= \prod_{P:\chi_2(P)=1}(1-\chi(P) u^{ \deg P})^{-1} \prod_{Q:\chi_2(Q)=-1}(1-\chi(Q)u^{ \deg Q})^{-1} (1-\chi(T)u)^{-1}, \\
L(u,\chi \cdot \chi_2) &= \prod_{P:\chi_2(P)=1}(1-\chi(P) u^{ \deg P})^{-1} \prod_{Q:\chi_2(Q)=-1}(1+\chi(Q)u^{ \deg Q})^{-1}.
\end{split}
\end{equation}
In particular, \eqref{lfunceuler} implies that
\begin{equation}\label{lfunceuler2}
\begin{split}
\frac{L(u,\chi)}{L(u,\chi \cdot \chi_2)} &= \prod_{Q:\chi_2(Q)=-1} \frac{(1-\chi(Q) u^{ \deg Q})^{-1}}{(1+\chi(Q) u^{ \deg Q})^{-1}} (1-\chi(T)u)^{-1}
\end{split}
\end{equation}
and that
\begin{equation}\label{lfunceuler3}
\begin{split}
\sqrt{ L(u,\chi)L(u,\chi \cdot \chi_2)} &=\prod_{P:\chi_2(P)=1} (1-\chi(P)u^{\deg P})^{-1} \prod_{Q:\chi_2(Q)=-1} (1-\chi^2(Q) u^{ 2\deg Q})^{-1/2} (1-\chi(T)u)^{-1/2}.
\end{split}
\end{equation}
Using \eqref{lfunceuler}--\eqref{lfunceuler3}, we find that the right hand side of \eqref{eulerprodb} factors as
\begin{equation}\label{eulerprodcomp}
\begin{split}
\prod_{P:\chi_2(P)=1}(1-\chi(P) u^{\deg P})^{-1} & \cdot \prod_{Q:\chi_2(Q)=-1}(1-\chi^2(Q) u^{2\deg Q})^{-1/2} \\
& \cdot \quad \prod_{i \ge 1} \prod_{Q:\chi_2(Q)=-1} \left( \frac{(1-\chi^{2^i}(Q) u^{2^i \deg Q})^{-1}}{(1+\chi^{2^i}(Q) u^{ 2^i \deg Q})^{-1}} \right)^{2^{-i-1}}\\
& \cdot (1-\chi(T)u)^{-1/2} \prod_{i\ge 1}(1-\chi^{2^i}(T)u^{2^i})^{-2^{-i-1}}\\
& \cdot (1-\chi(T)u)^{-1/2} \prod_{i \ge 1} (1-\chi^{2^i}(T)u^{2^i})^{2^{-i-1}} \\
&= \prod_{P:\chi_2(P)=1}(1-\chi(P) u^{\deg P})^{-1}  \cdot  \prod_{Q:\chi_2(Q)=-1}(1-\chi^2(Q) u^{2\deg Q})^{-1/2}\\
& \cdot   \prod_{i \ge 1} \prod_{Q:\chi_2(Q)=-1} \left( \frac{(1-\chi^{2^i}(Q) u^{2^i \deg Q})^{-1}}{(1+\chi^{2^i}(Q) u^{ 2^i \deg Q})^{-1}} \right)^{2^{-i-1}} (1-\chi(T)u)^{-1}.
\end{split}
\end{equation}
It remains to establish equality between the Euler products \eqref{eulerprodeasy} and \eqref{eulerprodcomp}. The contribution of the prime $T$ is the same in both, and so is the contribution of primes $P$ satisfying $\chi_2(P)=1$. Now let $Q$ be a prime satisfying $\chi_2(Q)=-1$. It is sufficient to prove that the contribution of this prime in both products is the same, that is
\begin{equation}\label{powseriden}
(1-\chi(Q^2)u^{2\deg Q})^{-1} = (1-\chi^2(Q) u^{2\deg Q})^{-1/2}  \prod_{i \ge 1}  \left( \frac{(1-\chi^{2^i}(Q) u^{2^i \deg Q})^{-1}}{(1+\chi^{2^i}(Q) u^{ 2^i \deg Q})^{-1}} \right)^{2^{-i-1}}.
\end{equation}
Letting $z=\chi^2(Q)u^{2\deg Q}$, the identity \eqref{powseriden} becomes
\begin{equation}\label{powseriden2}
(1-z)^{-1/2} = \prod_{i \ge 1} \left( \frac{1+z^{2^{i-1}}}{1-z^{2^{i-1}}} \right)^{2^{-i-1}},
\end{equation}
which follows by noting the telescoping nature of the right hand side of \eqref{powseriden2}:
\begin{equation*}
\prod_{i \ge 1} \left( \frac{1+z^{2^{i-1}}}{1-z^{2^{i-1}}} \right)^{2^{-i-1}} =\prod_{i \ge 1} \frac{(1-z^{2^i})^{2^{-i-1}}}{(1-z^{2^{i-1}})^{2^{-i}}} = (1-z)^{-1/2}.
\end{equation*} 
\end{proof}
We recall some basic facts from \cite[Sec. 6]{keating2016squarefree}. For a non-principal Dirichlet character $\chi$ modulo $Q$, the $L$-function $L(u,\chi)$ is a polynomial of degree at most $\deg(Q)-1$. The Riemann Hypothesis for $L(u,\chi)$ says that we may factor $L(u,\chi)$ as
\begin{equation}\label{eq:lfact}
L(u,\chi) = \prod_{i=1}^{\deg L(u,\chi)} (1-\gamma_i u), \qquad |\gamma_i| \le \sqrt{q}.
\end{equation}
In fact for non-trivial zeros $|\gamma_i| = q^{1/2}$, but $L(u,\chi)$ may have trivial zeros, in which case $|\gamma_i| = 1$.
\begin{lem}\label{lem:snbform}
Let $\chi_0 \neq \chi \colon \FF_q[T] \to \mathbb{C}$ be an even Dirichlet character modulo $T^{n+1}$ ($q$ odd). Then
\begin{equation}\label{snbchiviause}
S(n,b,\chi) = [u^n]\sqrt{L(u,\chi) L(u,\chi \cdot \chi_2)} + O_n(q^{\frac{n}{2} - \frac{1}{4}})
\end{equation}
and
\begin{equation}\label{snbchiviauseweak}
S(n,b,\chi) = O_n(q^{\frac{n}{2}}).
\end{equation}
\end{lem}
\begin{proof}
By Lemma \ref{snbchivia},
\begin{equation}\label{infprodstam}
\begin{split}
S(n,b,\chi) &= [u^n] \sqrt{L(u,\chi)L(u,\chi \cdot \chi_2)}\prod_{i\ge 1} \left( \frac{L(u^{2^i},\chi^{2^i})}{L(u^{2^i},\chi^{2^i} \cdot \chi_2)} \right)^{2^{-i-1}} (1-\chi(T)u)^{-1/2}\prod_{i \ge 1} (1-\chi^{2^i}(T)u^{2^i})^{2^{-i-1}}.
\end{split}
\end{equation}
Although the products in \eqref{infprodstam} are infinite, we may truncate them because only the coefficient of $u^n$ is of interest to us:
\begin{equation}\label{finiteprodstam}
\begin{split}
S(n,b,\chi) &=  [u^n] \sqrt{L(u,\chi)L(u,\chi \cdot \chi_2)}\prod_{i=1}^{n} \left( \frac{L(u^{2^i},\chi^{2^i})}{L(u^{2^i},\chi^{2^i} \cdot \chi_2)} \right)^{2^{-i-1}} (1-\chi(T)u)^{-1/2}\prod_{i=1}^{n} (1-\chi^{2^i}(T)u^{2^i})^{2^{-i-1}}.
\end{split}
\end{equation}
For any $i \ge 1$, the character $\chi^{2^i}$ is non-principal, since the order of $\chi$ (as a character of $(\FF_q[T]/T^{n+1})^{\times}$) is a power of $q$, and in particular it is odd. Hence, by making use of \eqref{eq:lfact} with $\chi^{2^i}$ and $\chi^{2^i} \cdot \chi_2$, we see that the $j$-th coefficients of $L(u,\chi^{2^i})$ and of $L(u,\chi^{2^i}\cdot \chi_2)$ are both of size $O_{j,n}(q^{j/2})$. In particular, for any $i \ge 1$,
\begin{equation}\label{2jest}
[u^j] L(u^{2^i},\chi^{2^i}), [u^j] L(u^{2^i},\chi^{2^i}\cdot \chi_2) = O_{j,n}(q^{j/2^{i+1}}) = O_{j,n}(q^{j/4}).
\end{equation}
From \eqref{2jest} we deduce that
\begin{equation}\label{infprodest}
[u^j] \prod_{i=1}^{n} \left( \frac{L(u^{2^i},\chi^{2^i})}{L(u^{2^i},\chi^{2^i} \cdot \chi_2)} \right)^{2^{-i-1}} (1-\chi(T)u)^{-1/2}\prod_{i=1}^{n} (1-\chi^{2^i}(T)u^{2^i})^{2^{-i-1}} = O_{j,n}(q^{j/4}).
\end{equation}
Additionally, from \eqref{eq:lfact} with $\chi$ and $\chi \cdot \chi_2$, we have $[u^j]L(u,\chi), [u^j]L(u,\chi \cdot \chi_2) =  O_{j,n}(q^{j/2})$, and so 
\begin{equation}\label{simpprodest}
[u^j]\sqrt{L(u,\chi) L(u,\chi \cdot \chi_2)} = O_{j,n}(q^{j/2}).
\end{equation}
Plugging the estimates \eqref{infprodest} and \eqref{simpprodest} in \eqref{finiteprodstam}, we establish \eqref{snbchiviause}. From \eqref{snbchiviause} and \eqref{simpprodest} with $j=n$, we obtain \eqref{snbchiviauseweak}. 
\end{proof}

\subsection{The passage to zeros of \texorpdfstring{$L$}{L}-functions}\label{sec:short_to_zeros}
We recall some facts from \cite[Sec. 6]{keating2016squarefree}. A Dirichlet character $\chi$ modulo $Q$ is primitive if there is no proper divisor $Q_0 \mid Q$ such that $\chi(F)=1$ whenever $F$ is coprime to $Q$ and $F \equiv 1 \bmod Q_0$. If $\chi$ is a primitive character modulo $Q$, then by the Riemann Hypothesis for $L(u,\chi)$ we have
\begin{equation}
L(u,\chi)=(1-u)^{a(\chi)} \prod_{i=1}^{\deg(Q)-1-a(\chi)}(1-\gamma_i(\chi)u),
\end{equation}
where $a(\chi)=1$ if $\chi$ is even and $a(\chi)=0$ otherwise, and $|\gamma_i(\chi)| = \sqrt{q}$. The unitarized Frobenius matrix of $\chi$ is the diagonal unitary matrix
\begin{equation}
\Theta_{\chi} := \mathrm{Diag}(\frac{\gamma_1(\chi)}{\sqrt{q}}, \ldots, \frac{\gamma_{\deg(Q)-1-a(\chi)}}{\sqrt{q}} ).
\end{equation}
\begin{proposition}\label{thmb}
Let $0 \le h \le n-1$. We have
\begin{multline}\label{eq:shortvariance_to_zeros}
\frac{\Var_{A \in \mathcal{M}_{n,q}} \big(\nu_{b}(A; h)\big)}{q^{h+1}} \\ = \frac{1}{q^{n-h-1}}\sum_{\substack{\chi \bmod T^{n-h} \\ \text{ even and primitive}}}\left| [u^n] \sqrt{\det(I-u\Theta_{\chi}) \cdot \det(I-u\Theta_{\chi \cdot \chi_2})} \right|^2  \\ +O_n\left( q^{-\frac{1}{4}} \right).
\end{multline}
\end{proposition}

\begin{proof}
By Lemma \ref{lem:snbform},
\begin{equation}\label{eq:summ}
\sum_{m=0}^{n}  S(m,b,\chi) = [u^n]\sqrt{L(u,\chi) L(u,\chi \cdot \chi_2)} + O_n(q^{\frac{n}{2} - \frac{1}{4}}) = O_n(q^{\frac{n}{2}}).
\end{equation}
From \eqref{eq:summ} and 
Lemma \ref{lemshortvar} we obtain
\begin{equation}\label{varland}
\Var_{A \in \mathcal{M}_{n,q}} \nu_{b}(A; h) = \frac{\sum_{\substack{\chi \bmod T^{n-h}\\\chi_0 \neq  \chi \text{ even}}} \left|[u^n]\sqrt{L(u,\chi) L(u,\chi \cdot \chi_2)} + O_n(q^{\frac{n}{2} - \frac{1}{4}}) \right|^2 }{q^{2(n-h-1)}}.
\end{equation}
Since the number of non-primitive even characters modulo $T^{n-h-1}$ is $O(q^{n-h-2})$ while the number of primitive even characters modulo $T^{n-h-1}$ is $q^{n-h-1}+O(q^{n-h-2})$ \cite[Sec. 6]{keating2016squarefree}, \eqref{eq:summ} and \eqref{varland} imply that
\begin{equation}\label{varland2}
\begin{split}
\Var_{A \in \mathcal{M}_{n,q}} \nu_{b}(A; h) &= \frac{\sum_{\substack{\chi \bmod T^{n-h} \\ \text{ even and primitive}}}\left|[u^n]\sqrt{L(u,\chi) L(u,\chi \cdot \chi_2)} + O_n(q^{\frac{n}{2} - \frac{1}{4}}) \right|^2 }{q^{2(n-h-1)}} + O_n(q^h)\\
& = \frac{\sum_{\substack{\chi \bmod T^{n-h} \\ \text{ even and primitive}}}\left|[u^n]\sqrt{L(u,\chi) L(u,\chi \cdot \chi_2)} \right|^2 }{q^{2(n-h-1)}} + O_n(q^{h-\frac{1}{4}}).
\end{split}
\end{equation}
We now write $L(u,\chi)$ as $\det(I-u\sqrt{q}\Theta_{\chi})(1-u)$ and $L(u,\chi\cdot \chi_2)$ as $\det(I-u\sqrt{q}\Theta_{\chi \cdot \chi_2})$ to obtain 
\begin{equation}\label{varland3}
\Var_{A \in \mathcal{M}_{n,q}} \nu_{b}(A; h) = \frac{\sum_{\substack{\chi \bmod T^{n-h} \\ \text{ even and primitive}}}\left|[u^n]\sqrt{\det(I-u\sqrt{q} \Theta_{\chi})\det(I-u\sqrt{q} \Theta_{\chi})(1-u)} \right|^2 }{q^{2(n-h-1)}} + O_n(q^{h-\frac{1}{4}}).
\end{equation}
The proof is concluded by writing in \eqref{varland3} $[u^n]\sqrt{\det(I-u\sqrt{q} \Theta_{\chi})\det(I-u\sqrt{q} \Theta_{\chi})(1-u)} = q^{n/2} [u^n] \sqrt{\det(I-u\Theta_{\chi}) \det(I-u\Theta_{\chi \cdot \chi_2})} + O_n(q^{(n-1)/2})$ and dividing both sides by $q^{h+1}$.
\end{proof}

\section{Equidistribution and random matrix integrals}\label{sec:equidistr}

We turn to an evaluation of the average in the formula \eqref{eq:shortvariance_to_zeros}. We are able to evaluate these averages by making use of a recent equidistribution theorem of Sawin \cite{sawin2018equidistribution}. We adopt the following notation for a continuous class function $f\colon \prod_{i=1}^r U(N_i) \rightarrow \mathbb{C}$, where $U(n)$ is the $n\times n$ unitary group; we define the function $\langle f \rangle$ as the unique continuous function mapping $U(1)^r \rightarrow \mathbb{C}$ such that
\begin{multline*}
\int_{\prod U(N_i)} f(g_1,...,g_r) \psi(\det g_1, ..., \det g_r)\, dg_1\cdots dg_r \\
 = \int_{\prod U(N_i)} \langle f \rangle(\det g_1, ..., \det g_r) \psi(\det g_1, ..., \det g_r)\, dg_1\cdots dg_r,
\end{multline*}
for all continuous functions $\psi\colon U(1)^r \rightarrow \mathbb{C}$. That is, $\langle f \rangle (c_1,...,c_r)$ is the integral of $f$ over the coset of $\prod SU(N_i) \le \prod U(N_i)$ consisting of elements with determinants $c_1,...,c_r$, against the unique $\prod SU(N_i)$-invariant measure on that coset, of total mass 1.

As a special case of the aforementioned result of Sawin \cite[Theorem 1.2]{sawin2018equidistribution}, one has that
\begin{thm}
\label{UtimesU_equi}
If $f\colon U(N-1)\times U(N)$ is a continuous class function and $N \geq 6$, then
$$
\lim_{q\rightarrow\infty} \bigg[\frac{1}{q^N} \sum_{\substack{\chi \;(T^{N+1}) \\ \mathrm{ev., prim.}}} f(\Theta_\chi, \Theta_{\chi\cdot \chi_2}) - \frac{1}{q^N}  \sum_{\substack{\chi \;(T^{N+1}) \\ \mathrm{ev., prim.}}} \langle f\rangle (\det \Theta_\chi, \det \Theta_{\chi\cdot \chi_2}) \bigg] = 0
$$
where the limit is taken for $q$ of fixed characteristic.
\end{thm}

\begin{remark}
Note that $(\Theta_\chi,\Theta_{\chi\cdot \chi_2}) \in U(N-1) \times U(N)$ even though both $L(u,\chi)$ and $L(u,\chi\cdot \chi_2)$ have $N$ zeros. The reason is that for even $\chi$, $L(u,\chi)$ has only $N-1$ non-trivial zeros; on the other hand $\chi \cdot \chi_2$ is no longer even and $L(u,\chi\cdot \chi_2)$ has $N$ non-trivial zeros.
\end{remark}

We introduce the notation, for a unitary matrix $g$,
\begin{equation}
\label{A_kz}
A_{k,(z)}(g) := [u^k] \det(1-ug)^z.
\end{equation}
Note that $A_{k,(z)}(g)$ is a symmetric homogeneous polynomial of degree $k$ in the eigenvalues of $g$. Because we will make use especially of the case $z=1/2$, we introduce the abbreviation $A_k(g):= A_{k,(1/2)}(g)$.

Theorem \ref{UtimesU_equi} allows us to deduce the following.
\begin{cor}\label{sqrt_char_to_matrix}
Fix constants $- 1\leq h \le n-1$ and let $N = n-h-1$. For $n \leq N(N-1)$ and $N \geq 6$,
\begin{multline}\label{eq:sqrt_char_to_matrix}
\lim_{q\rightarrow\infty} \frac{1}{q^{n-h-1}} \sum_{\substack{\chi\; (T^{n-h}) \\ \mathrm{ev., prime.}}}\Big| [u^n]\sqrt{\det(1- u \Theta_\chi) \det(1-u \Theta_{\chi\cdot \chi_2})}\Big|^2 
\\ = \sum_{\substack{j+k=n \\ j,k \geq 0}} \int_{U(N-1)} |A_j(g_1)|^2\, dg_1 \int_{U(N)} |A_k(g_2)|^2\, dg_2,
\end{multline}
with the limit is taken along a sequence of $q$ of fixed characteristic.
\end{cor}
We remark that the right hand side of \eqref{eq:sqrt_char_to_matrix} may be seen to be equal to
$$
\int_{\substack{g_1 \in U(N-1) \\ g_2 \in U(N)}} \Big| [u^n] \sqrt{\det(1-u g_1) \det(1 - u g_2)}\Big|^2 \, dg_1 dg_2.
$$
In order to prove Corollary \ref{sqrt_char_to_matrix} we draw upon the following:

\begin{lem}
\label{coset_cancellation}
Let $f\colon U(N)\rightarrow\mathbb{C}$ be a function such that for $g\in U(N)$, $f(g)$ is a symmetric homogeneous Laurent polynomial of degree $k$ in the eigenvalues of $g$. The following hold.
\begin{enumerate}[i)]
\item If $k \not\equiv 0 \bmod N$, then
$$
\langle f \rangle (z) = 0, \quad \textrm{for all }\; |z|=1,
$$
\item If $k=0$, then
$$
\langle f \rangle (z)  = \int_{U(N)} f(g)\, dg, \quad \textrm{for all }\; |z|=1.
$$
\end{enumerate}
\end{lem}
\begin{proof}
Both i) and ii) make use of the following assertion: that if $F\colon U(N) \rightarrow\mathbb{C}$ is a symmetric homogeneous polynomial of degree $k \neq 0$ in the eigenvalues of a matrix from $U(N)$, then
\begin{equation}
\label{full_cancellation}
\int_{U(N)} F(g)\, dg = 0.
\end{equation}
For, the Haar measure is invariant under scalar multiplication, so for any $c \in U(1)$,
$$
0 = \int_{U(N)} F(g)\, dg - \int_{U(N)} F(cg)\, dg = (1-c^k) \int_{U(N)} F(g)\, dg.
$$
If $k \neq 0$, there exists $c \in U(1)$ such that $(1-c^k) \neq 0$ and \eqref{full_cancellation} follows.

Turning to i), note that this will be proved if we show for $k \not \equiv 0 \bmod N$ that
\begin{equation}
\label{part_cancellation}
\int_{U(N)} f(g) \psi(\det g)\, dg = 0,
\end{equation}
for all continuous $\psi\colon U(1) \rightarrow\mathbb{C}$. In turn by Fourier analysis, since $\det g \in U(1)$ for all $g \in U(N)$, to establish \eqref{part_cancellation} we need only establish it for $\psi(z) = z^\ell$ with $\ell \in \mathbb{Z}$. But if $f(g)$ is of degree $k$ in the eigenvalues of $g$, then 
$$
f(g) (\det g)^\ell
$$
is of degree $k+N\ell$. As $k \not\equiv 0 \bmod N$, we have $k + N\ell \neq 0$, and hence
$$
\int_{U(N)} f(g) (\det g)^\ell\, dg = 0,
$$
establishing the claim i).

For ii), our proof is similar. We must show
\begin{equation}
\label{splitting_cancellation}
\int_{U(N)} f(g) \psi(\det g)\, dg = \int_{U(N)} f(g)\, dg \, \int_{U(N)} \psi(\det g)\, dg.
\end{equation}
As before it suffices to verify this claim when $\psi(z) = z^\ell$. For $\ell = 0$ this is clear, and when $\ell \neq 0$, note that
$$
\int_{U(N)} (\det g)^\ell\, dg = 0,
$$
so that we establish \eqref{splitting_cancellation} by showing
$$
\int_{U(N)} f(g) (\det g)^\ell\, dg = 0.
$$
But as $f(g) (\det g)^\ell$ is of degree $\ell \neq 0$, this is indeed the case, establishing the claim.
\end{proof}

We now may return to Corollary \ref{sqrt_char_to_matrix}.

\begin{proof}[Proof of Corollary \ref{sqrt_char_to_matrix}]
Note that
$$
[u^n] \sqrt{\det(1-u\Theta_\chi) \det(1- u \Theta_{\chi\cdot \chi_2})} = \sum_{\substack{j+k=n \\ j, k \geq 0}} A_j(\Theta_\chi) A_k(\Theta_{\chi\cdot \chi_2}).
$$
Hence the left hand side of \eqref{eq:sqrt_char_to_matrix} is
\begin{equation}
\label{eq:sqrt_char_to_matrix_exp}
\lim_{q\rightarrow\infty} \frac{1}{q^{n-h-1}} \sum_{\substack{\chi\; (T^{n-h}) \\ \textrm{ev., prim.}}} \sum_{\substack{j+k = n \\ j,k \geq 0}}\sum_{\substack{j'+k' = n \\ j',k' \geq 0}} A_j(\Theta_\chi)A_k(\Theta_{\chi\cdot\chi_2}) \overline{A_{j'}(\Theta_\chi)A_{k'}(\Theta_{\chi\cdot\chi_2})}.
\end{equation}
We will need to evaluate the random matrix coset integral
$$
\langle A_j A_k \overline{A_{j'}A_{k'}}\rangle = \langle A_j\overline{A_{j'}} \rangle \langle A_k \overline{A_{k'}}\rangle.
$$
Note that if $j = j'$, then $k = k'$ also. Noting that $A_j(g) \overline{A_{j}(g)} = A_j(g) A_j(g^{-1})$ and likewise for $A_k$, one may see that $A_j \overline{A_j}$ and $A_k \overline{A_k}$ are homogeneous symmetric Laurent polynomials of degree $0$. Thus by Lemma \ref{coset_cancellation}, we have for all $|z|=1$, 
$$
\langle A_j \overline{A_j}\rangle(z) = \int_{U(N-1)} |A_j(g)|^2 \, dg,
$$
$$
\langle A_k \overline{A_k}\rangle(z) = \int_{U(N)} |A_k(g)|^2\, dg.
$$
Furthermore, in the sum \eqref{eq:sqrt_char_to_matrix_exp}, if $j\neq j'$ and $k\neq k'$, we may reason in the same way to see that $A_j \overline{A_{j'}}$ and $A_k\overline{A_{k'}}$ are homogeneous symmetric Laurent polynomials of non-zero degrees, say $\ell$ and $-\ell$ respectively, with $|\ell| \leq n < N(N-1)$. As no non-zero number smaller in magnitude than $N(N-1)$ is divisible by both $N$ and $N-1$, Lemma \ref{coset_cancellation} implies that one of 
$$
\int_{U(N-1)} A_j(g) \overline{A_{j'}(g)}\, dg = 0 \quad \textrm{or} \quad \int_{U(N)} A_k(g) \overline{A_{k'}(g)}\, dg = 0
$$
holds, so in particular the product is always $0$.

From this analysis it follows that for all matrices $\Theta_\chi$ and $\Theta_{\chi\cdot \chi_2}$
$$
\langle A_j A_k \overline{A_{j'} A_{k'}} \rangle(\det \Theta_\chi, \det \Theta_{\chi\cdot \chi_2}) = \begin{cases} \int_{U(N-1)}|A_j(g_1)|^2 dg_1 \cdot \int_{U(N)} |A_k(g_2)|^2\, dg_2 & \textrm{if}\; j = j' \\ 0 & \textrm{otherwise} \end{cases}.
$$
Thus using Theorem \ref{UtimesU_equi}, \eqref{eq:sqrt_char_to_matrix_exp} simplifies to
$$
\sum_{\substack{j+k=n \\ j,k \geq 0}} \int_{U(N-1)} |A_j(g_1)|^2\, dg_1 \int_{U(N)} |A_k(g_2)|^2\, dg_2,
$$
as claimed.
\end{proof}

\section{\texorpdfstring{$z$}{z}-measures on partitions}\label{sec:zmeasures}

\subsection{Definitions}\label{sec:zdefinitions}
In order to give a succinct evaluation of the integrals on the right hand side of Corollary \ref{sqrt_char_to_matrix}, we make use of $z$-measures on partitions, first introduced by Kerov, Olshanski, and Vershik (in \cite{kerov1993harmonic}). These can be thought of as a generalization of Plancherel measures. We give a short introduction here; a more thorough introduction can be found in \cite{olshanski2003introduction}. The $z$-measures are a two-parameter family of measures on partitions, though it is often natural to specialize to a one-parameter subfamily. In order to define the $z$-measures we make use of standard notation in enumerative combinatorics, along the lines of e.g. \cite[Ch. 7]{stanley1999enumerative}. We view partitions $\lambda \vdash n$ as Young diagrams with $n$ boxes. Recall (from e.g. \cite[Sec. 7.21]{stanley1999enumerative}) that for a square $\square$ in $\lambda$ with position $(i,j)$ (where $1 \leq j \leq \lambda_i$), the \emph{content} $c(\square)$ is defined by
\begin{equation*}
c(\square) = j-i.
\end{equation*}
We let $\dim(\lambda)$ be the dimension of the irreducible representation of $S_n$ associated to the partition $\lambda$; equivalently $\dim(\lambda)$ is equal to the number of standard Young tableaux of shape $\lambda$. The $z$-measure on partitions of $n$ with parameters $z$ and $z'$, written $M^{(n)}_{z,z'}$ is the measure on the set of all partitions $\lambda$ of $n$ satisfying
\begin{equation}
\label{z_measure_def}
M^{(n)}_{z,z'}(\lambda):= \frac{\dim(\lambda)^2}{n! (z z')_n} \prod_{\square \in \lambda} (z+c(\square))(z'+c(\square)).
\end{equation}
Recall that $(x)_j := (x)(x+1)\cdots (x+j-1)$ is the rising factorial Pochhammer symbol. The expression  \eqref{z_measure_def} is well defined for all $z, z' \in \mathbb{C}$ with $zz' \notin \mathbb{Z}_{\leq 0}$. Furthermore we use the convention that $\varnothing$ is the sole partition of $0$ and for any $z,z'$,
\begin{equation}
M_{z,z'}^{(0)}(\varnothing) = 1.
\end{equation}
For any $n$ and $z,z' \in \mathbb{C}$ with $zz' \notin \mathbb{Z}_{\leq 0}$ one has
\begin{equation}
\sum_{\lambda \vdash n} M_{z,z'}^{(n)}(\lambda) = 1,
\end{equation}
though this fact is not obvious (see e.g. \cite{okounkov2001sl} for a proof). It is not always the case that $M_{z,z'}^{(n)}(\lambda) \geq 0$ for all $\lambda$ (so that in some cases $M_{z,z'}^{(n)}$ must be viewed as a signed measure) but when, for instance, $z' = \overline{z}$, plainly \eqref{z_measure_def} is always non-negative. 

Note from the definition \eqref{z_measure_def}, for fixed $n$, this measure tends toward the Plancherel measure as $z,z'\rightarrow\infty$.

We denote $M_z^{(n)}(\lambda):= M^{(n)}_{z,\overline{z}}(\lambda)$, and moreover for a subset $A$ of the set of all partitions of $n$, we use the notations
\begin{equation}
\mathbb{P}_{z,z'}^{(n)}(\lambda \in A) = \sum_{\lambda \in A} M_{z,z'}^{(n)}(\lambda), \quad \text{and}\quad \mathbb{P}_z^{(n)}(\lambda \in A) = \sum_{\lambda \in A} M_z^{(n)}(\lambda).
\end{equation}

It is known that there exists a scaling limit of the $z$-measures as $n\rightarrow\infty$; these scaling limits were first investigated as a part of representation theory on the infinite symmetric group. We do not review the full theory here, instead referring the reader to \cite{olshanski2003introduction} for an introduction. The result from this theory that we will make use of is
\begin{thm}
\label{thm:z_scalinglimit}
For any $z \in \mathbb{C}\setminus \mathbb{Z}_{\leq 0}$, there exists a random variable $\alpha_1^{(z)}$ lying almost surely in the interval $[0,1]$ such that for $\lambda \vdash n$ chosen according to the $z$-measure with parameters $z, \overline{z}$ we have
\begin{equation}
\label{z_limiting}
\lim_{n\rightarrow\infty}\mathbb{P}\Big(\frac{\lambda_1}{n} \leq x\Big) = \mathbb{P}(\alpha_1^{(z)} \leq x),
\end{equation}
for all real $x$. 

Moreover for $z \in \mathbb{C}\setminus\mathbb{Z}_{\leq 1}$ with $z' = \overline{z}$ as above, the function $F_z(x) = \mathbb{P}(\alpha_1^{(z)} \leq x)$
is continuous for all $x \in \mathbb{R}$.
\end{thm}

We simply take this theorem as our definition of $\alpha_1^{(z)}$ -- that is, $\alpha_1^{(z)}$ is the random variable with distribution function given by this limit -- but we note that there exists a more sophisticated perspective in which the random variable $\alpha_1^{(z)}$ is \emph{the largest part of the $z$-measure point process with parameters $z, \overline{z}$ on the Thoma simplex}; see again \cite{borodin1998point} for more about this latter object and its connection to the infinite symmetric group. We adopt the notational convention that $\alpha_1 = \alpha_1^{(1/2)}$.

Theorem \ref{thm:z_scalinglimit} as written does not directly appear in the literature, but it can be proved by piecing together several results proved in the papers \cite{olshanski1998point,borodin1998point,borodin2005z}. We outline the proof of the theorem from these pieces in Appendix \ref{sec:Further}.

\begin{remark}
\label{z_remark}
The theorem above does not treat the case $z \in \mathbb{Z}_{\leq 0}$ and does not fully treat $z=1$. These cases will not be necessary for us in what follows but in fact their limit can be analyzed directly from the definition \eqref{z_measure_def}. Still with $z' = \overline{z}$, observe the following. For $z \in \mathbb{Z}_{\leq 0}$ then with probability $1$ we have $\lambda_1 \leq |z|$ under the $z$-measure, and so $\lambda_1/n \rightarrow 0$. For $z=1$, for $\lambda \vdash n$ under the $z$-measure one must have with probability $1$ that $\lambda_1 = n$ and so $\lambda_n/n \rightarrow 1.$
\end{remark}

\subsection{The evaluation of random matrix integrals}\label{sec:RMT}
We need the following result in order to evaluate random matrix integrals such as those appearing in Corollary \ref{sqrt_char_to_matrix}.

\begin{thm}\label{RMT_to_zmeasures}
For $g \in U(N)$, with $A_{n,(z)}(g)$ defined by \eqref{A_kz}, we have
\begin{equation}\label{equal_indices}
\int_{U(N)} A_{n,(z)}(g) A_{n,(z')}(g^{-1})\, dg = \frac{(zz')_n}{n!} \mathbb{P}_{z,z'}^{(n)}(\lambda_1 \leq N).
\end{equation}
Furthermore for $n\neq m$, we have
\begin{equation}\label{unequal_indices}
\int_{U(N)} A_{n,(z)}(g) A_{m,(z')}(g^{-1})\, dg = 0.
\end{equation}
\end{thm}

\begin{remark}
As noted before, we have used the convention that $\mathbb{P}_{z,z'}^{(0)}(\lambda_1 \leq N) = 1$ for all $N$ and $z,z'$, so that the above identities make sense even for $n=0$.
\end{remark}

We use Schur functions to prove this theorem. For $x = (x_1,...,x_N)$ and $\lambda$ a partition, we use the notation $s_\lambda(x)$ to denote the Schur function of shape $\lambda$ (see \cite[Ch. 7]{stanley1999enumerative}). 

We will use of the following well-known results:

First, we recall the \emph{dual Cauchy identity} (see \cite[Thm 7.14.3]{stanley1999enumerative}),
\begin{equation}\label{dual_cauchy}
\prod_{i,j} (1+ x_i y_j) = \sum_\lambda s_\lambda(x) s_{\lambda'}(y),
\end{equation}
where $\lambda'$ is the dual partition to $\lambda$.

Second, we recall the following evaluation of Schur functions (proved by combining \cite[Cor 7.21.4]{stanley1999enumerative} and \cite[Cor 7.21.6]{stanley1999enumerative}),
\begin{equation}\label{schur_eval}
s_\lambda(\underbrace{1,...,1}_{k}) = \frac{\dim(\lambda)}{n!} \prod_{\square\in\lambda}(k+c(\square)),
\end{equation}
for $\lambda \vdash n$.

Third, we recall the orthogonality relations for Schur functions in eigenvalues of the unitary group (see e.g. \cite{bump2004lie}). If $g \in U(N)$ has eigenvalues $y_1,...,y_N$ and we use the Schur function notation $s_\lambda(g):= s_\lambda(y_1,...,y_N)$, for any two partitions $\lambda$ and $\nu$,
\begin{equation}\label{schur_orthogonality}
\int_{U(N)} s_\lambda(g) s_\nu(g^{-1})\, dg = \delta_{\lambda = \nu, \ell(\lambda) \leq N}.
\end{equation}

\begin{proof}[Proof of Theorem \ref{RMT_to_zmeasures}]
We start by specializing to the case where $z$ is a positive integer; later on we will consider more general $z$. We make use of the dual Cauchy identity \eqref{dual_cauchy} in the variables $x_1,...,x_z$ and $y_1,...,y_N$ where for all $i$, $x_i = -u$, and $y_1,...,y_N$ are the $N$ eigenvalues  of $g \in U(N)$. The dual Cauchy identity tells that
\begin{equation}
\det(1-ug)^z = \sum_\lambda s_\lambda(-u,...,-u) s_{\lambda'}(g) = \sum_\lambda (-u)^{|\lambda|} s_\lambda(\underbrace{1,...,1}_z) s_{\lambda'}(g).
\end{equation}
Yet from \eqref{schur_eval}, we see we can write this as
\begin{equation}\label{integer_z_only}
\det(1-ug)^z = \sum_{n=0}^\infty u^n \Big((-1)^n \sum_{\lambda \vdash n} \frac{\dim(\lambda)}{n!} \prod_{\square \in \lambda} (z+c(\square)) s_{\lambda'}(g)\Big),
\end{equation}
where we adopt the convention that the coefficient for $n=0$ is $1$. Note that we have so far only proved \eqref{integer_z_only} for positive integer $z$.

For $|u| < 1$, the binomial series tells us that
\begin{equation*}
(1-uy_i)^z = \sum_{n=0}^\infty u^n \Big( (-y_i)^n \frac{(z)_n}{n!}\Big),
\end{equation*}
for all complex $z$. In particular the coefficients of $u^n$ in this series are polynomials in $z$. Multiplying $N$ such identities, it follows that for $|u|< 1$ and all complex $z$,
\begin{equation*}
\det(1-ug)^z = \prod_{i=1}^N (1-u y_i)^z = \sum_{n=0}^\infty u^n P_{n,g}(z),
\end{equation*}
where $P_{n,g}(z)$ are polynomials in $z$. From \eqref{integer_z_only} we obtain the expression
\begin{equation}\label{polynomial_identity}
P_{n,g}(z) = (-1)^n \sum_{\lambda \vdash n} \frac{\dim(\lambda)}{n!} \prod_{\square \in \lambda} (z+c(\square)) s_{\lambda'}(g),
\end{equation}
valid for positive integer $z$. But as both the left and right hand sides are polynomials in $z$ equal at all positive integers, it follows that this identity holds for all $z\in \mathbb{C}$.

But of course, $P_{n,g}(z) = A_{n,(z)}(g)$, so that using \eqref{polynomial_identity} and orthogonality relations \eqref{schur_orthogonality} for Schur functions,
\begin{align*}
\int_{U(N)} A_{n,(z)}(g) A_{n,(z')}(g^{-1})\, dg &= \sum_{\lambda \vdash n} \frac{(\dim \lambda)^2}{(n!)^2} \prod_{\square \in \lambda} (z+c(\square))(z'+c(\square)) \delta_{\ell(\lambda')\leq N} \\
&= \frac{(zz')_n}{n!} \mathbb{P}_{z,z'}^{(n)}(\lambda_1 \leq N).
\end{align*}
This verifies \eqref{equal_indices}. 

By contrast \eqref{unequal_indices} is much simpler; $A_{n,(z)}$ and $A_{m,(z)}$ are symmetric homogeneous polynomials of degree $n$ and $m$ respectively. Using the invariance of Haar measure on the unitary group under scalar multiplication, we have
\begin{multline*}
\int_{U(N)} A_{n,(z)}(g) A_{m,(z')}(g^{-1})\, dg = \int_{U(N)} A_{n,(z)}(\omega g) A_{m,(z')}(\omega^{-1} g^{-1})\, dg \\ = \omega^{n-m} \int_{U(N)} A_{n,(z)}(g) A_{m,(z')}(g^{-1})\, dg,
\end{multline*}
for any $|\omega|=1$. But if $n\neq m$, this can only be the case if this integral vanishes; that is, \eqref{unequal_indices} holds.
\end{proof}

\section{Arithmetic functions and \texorpdfstring{$z$}{z}-measures}\label{sec:arith_zmeasures}

\subsection{On Theorem \ref{main_thm_Variance}}\label{sec:mainthm_proof}
We are now able to prove our main result Theorem \ref{main_thm_Variance}. Indeed, it follows quickly by combining results proved above.
\begin{proof}[Proof of Theorem \ref{main_thm_Variance}]
We note from Proposition \ref{thmb} and Corollary \ref{sqrt_char_to_matrix} that for $N = n-h-1$,
\begin{equation*}
\Var_{A \in \mathcal{A}_{n,q}}( \nu_b(A;h)) = q^{h+1} \sum_{\substack{j+k=n \\ j,k \geq 0}} \int_{U(N-1)} |A_j(g_1)|^2\, dg_1 \int_{U(N)} |A_k(g_2)|^2\, dg_2 + o(q^{h+1}),
\end{equation*}
for $n \leq N(N-1)$ and $0 \leq h \leq n-7$ (the upper bound restriction comes from requiring that $N \geq 6$ in Corollary \ref{sqrt_char_to_matrix}). But then these integrals are evaluated using Theorem \ref{RMT_to_zmeasures} with $z = z' = 1/2$, and the result is Theorem \ref{main_thm_Variance}.
\end{proof}

\subsection{On Theorem \ref{d_z_variance}}\label{sec:dz_proof}
The idea behind this theorem, as in the proof of Theorem \ref{main_thm_Variance}, is to use an equidistribution theorem for the zeros of $L$-functions to relate the variance we seek to compute to integrals $\int_{U(N-1)} |[u^n] \det(1-ug)^z|^2 \, dg$. In this case however we can rely upon results already in the literature -- we make use of the set up in \cite{rodgers2018arithmetic}.

\begin{proof}[Proof of Theorem \ref{d_z_variance}]
From \eqref{d_z_def}, we see that if $f$ is squarefree with $\omega(f)$ the number of distinct prime factors of $f$ (so $f = P_1\cdots P_{\omega(f)}$, with all $P_i$ distinct primes), then
\begin{equation}\label{z_tothe_omega}
d_z(f) = z^{\omega(f)}.
\end{equation}
One the other hand, by Proposition 9.6 of \cite{rodgers2018arithmetic}, if $f \in \mathcal{M}_{n,q}$ is squarefree,
$$
d_z(f) = \sum_{\lambda \vdash n} s_\lambda(\underbrace{1,...,1}_z) X^\lambda(f),
$$
where $X^\lambda(f)$ is a character of $S_n$ applied to the factorization type of $f$ (as opposed to the cycle types of $S_n$ as usual; see \cite{rodgers2018arithmetic} for a further explanation). As before \eqref{schur_eval} implies that this may be written
\begin{equation}\label{z_tothe_omega2}
d_z(f) = \sum_{\lambda \vdash n}\frac{\dim(\lambda)}{n!} \prod_{\square\in\lambda}(z+c(\square)) X^\lambda(f).
\end{equation}
But this equals just $z^\omega(f)$, and since both are for a fixed $f$ polynomials in $z$ agreeing in value for all integer $z$, we can conclude that they agree for all $z\in \mathbb{C}$. Hence for $f$ squarefree, and any $z \in \mathbb{C}$, we see that \eqref{z_tothe_omega2} holds. Furthermore, for $f$ that is not squarefree, we have
\begin{equation*}
d_z(f) = \sum_{\lambda \vdash n}\frac{\dim(\lambda)}{n!} \prod_{\square\in\lambda}(z+c(\square)) X^\lambda(f) + c(f),
\end{equation*}
where $c(f)$ is some function (depending on $z$) supported on elements $f$ of $\mathcal{M}_{n,q}$ that are not squarefree.

Now we note that Theorem \ref{d_z_variance} follows directly from Theorem 3.1 of \cite{rodgers2018arithmetic}.
\end{proof}
\begin{remark}
The coefficients of $X^{\lambda}$ in \eqref{z_tothe_omega2} are called the Fourier coefficients of $d_z$, and are uniquely determined once $q$ is large enough with respect to $n$ (namely $q \ge n$). A feature of $d_z$ for $z$ not an integer is that all its Fourier coefficients are non-zero, while for integer $z$ only polynomially many (in $n$) coefficients are non-zero.
\end{remark}

\subsection{Limiting distributions: Propositions \ref{T_limit} and \ref{Thoma_limit}}\label{scalinglimit_proof}

It is natural to treat Proposition \ref{Thoma_limit} first and then Proposition \ref{T_limit}. In order to make claims regarding limits more transparent, in this section we write $\mathbb{P}(\alpha_1^{(z)} \leq s)$ rather than the abbreviation $F_z(s)$.
\begin{proof}[Proof of Prop. \ref{Thoma_limit}]
As
\begin{equation*}
\frac{(|z|^2)_n}{n!} = \frac{n^{|z|^2-1}}{\Gamma(|z|^2)} + o(n^{|z|^2-1}),
\end{equation*}
for $z\neq 0$ and (using the continuity of $\mathbb{P}(\alpha_1^{(z)} \leq s)$ in $s$),
\begin{equation*}
\mathbb{P}_z^{(n)}(\lambda_1 \leq N-1) = \mathbb{P}_z^{(n)}\Big(\frac{\lambda_1}{n} \leq \frac{N-1}{n}\Big) = \mathbb{P}(\alpha_1^{(z)} \leq s) + o(1),
\end{equation*}
the result follows.
\end{proof}

\begin{proof}[Proof of Prop. \ref{T_limit}]
Note that we have for large $j$,
\begin{equation*}
\frac{(1/4)_j}{j!} = \frac{j^{-3/4}}{\Gamma(1/4)} + o_{j\rightarrow\infty}(j^{-3/4}),
\end{equation*}
and $\frac{(1/4)_j}{j!} = O(j^{-3/4})$ in general for $j \ge 1$.

Likewise for large $j$,
\begin{equation*}
\mathbb{P}_{1/2}^{(j)}(\lambda_1 \leq N) = \mathbb{P}(\alpha_1 \leq N/j) + o_{j\rightarrow\infty}(1).
\end{equation*}
This convergence is uniform as $N$ varies (because $\mathbb{P}_{1/2}^{(j)}(\lambda_1 \leq N) = 1 = \mathbb{P}(\alpha_1 \leq N/j)$ for $N \geq j$ and otherwise $N/j$ lies in a compact interval). Furthermore, we have $\mathbb{P}_{1/2}^{(j)}(\lambda_1 \leq N-1) = O(1)$ in general.

Fix an arbitrary $\epsilon \in (0,1)$, and decompose
\begin{align}
\label{T_decomp}
\notag T(n;N) =& \Big( \sum_{\epsilon n \leq j \leq (1-\epsilon)n} + \sum_{\substack{j < \epsilon n, \, \text{or} \\ j > (1-\epsilon) n}}\Big)  \frac{(1/4)_j (1/4)_{n-j}}{j! (n-j)!} \mathbb{P}_{1/2}^{(j)}(\lambda_1 \leq N-1) \mathbb{P}_{1/2}^{(n-j)}(\lambda_1 \leq N)\\
\notag =& \frac{1}{\Gamma(1/4)^2} \sum_{\epsilon n \leq j \leq (1-\epsilon) n} j^{-3/4} (n-j)^{-3/4} \mathbb{P}(\alpha_1 \leq \frac{N-1}{j}) \mathbb{P}(\alpha_1 \leq \frac{N}{n-j}) \\
&+ o_{n\rightarrow\infty}\Big(\sum_{\epsilon n \leq j \leq (1-\epsilon) n} j^{-3/4} (n-j)^{-3/4}\Big)+ O\Big(\sum_{\substack{0 < j < \epsilon n, \, \text{or} \\ n >j > (1-\epsilon) n}} j^{-3/4}(n-j)^{-3/4}\Big) +O(n^{-3/4}),
\end{align}
where above the rate at which the error term $o_{n\rightarrow\infty}(\cdots)$ tends to zero as $n\rightarrow\infty$ depends upon $\epsilon$, but the constants of other error terms are absolute, with the last error term $O(n^{-3/4})$ coming from the terms $j=0$ and $j=n$ in the sum. If $N/n \rightarrow s$ as $n\to \infty$, then
\begin{equation*}
\mathbb{P}(\alpha_1 \leq \frac{N-1}{j}) = \mathbb{P}(\alpha_1 \leq \frac{(N-1)/n}{j/n}) = \mathbb{P}(\alpha_1 \leq \frac{s}{j/n}) + o_{n\rightarrow\infty}(1),
\end{equation*}
uniformly for $1 \leq j \leq  n$. (The reason for uniformity is again due to compactness.) Of course we have 
\begin{equation*}
\mathbb{P}(\alpha_1 \leq \frac{N}{n-j}) = \mathbb{P}(\alpha_1 \leq \frac{s}{1-j/n}) + o_{n\rightarrow\infty}(1)
\end{equation*}
also. 

Moreover,
\begin{equation*}
\sum_{\substack{0 < j < \epsilon n, \, \text{or} \\ n>j > (1-\epsilon) n}} j^{-3/4}(n-j)^{-3/4} = O(\epsilon^{1/4} n^{-1/2}),
\end{equation*}
and
\begin{equation*}
\sum_{0 < j < n} j^{-3/4} (n-j)^{-3/4} = O(n^{-1/2})
\end{equation*}
Hence the reader should check that we can simplify \eqref{T_decomp} to
\begin{align}
\notag (\ref{T_decomp}) =& \frac{1}{\Gamma(1/4)^2} \frac{1}{\sqrt{n}} \sum_{\epsilon n \leq j \leq (1-\epsilon) n} \frac{1}{n} (j/n)^{-3/4} (1-j/n)^{-3/4} \mathbb{P} (\alpha_1 \leq \frac{s}{j/n}) \mathbb{P}(\alpha_1 \leq \frac{s}{1-j/n}) \\
& + O(\epsilon^{1/4}n^{-1/2}) + o_{n\rightarrow\infty}(n^{-1/2}) \\
\notag = &\frac{1}{\Gamma(1/4)^2} \frac{1}{\sqrt{n}} \int_\epsilon^{1-\epsilon} t^{-3/4} (1-t)^{-3/4} \mathbb{P}(\alpha_1 \leq \frac{s}{t}) \mathbb{P}(\alpha_1 \leq \frac{s}{1-t})\, dt\\
& + O(\epsilon^{1/4} n^{-1/2}) + o_{n\rightarrow\infty}(n^{-1/2}),
\end{align}
with the second line following because the sum in the previous line is a Riemann sum. Completing the integral from the interval $[\epsilon, 1-\epsilon]$ to $[0,1]$ adds only an error of $O(\epsilon^{1/4} n^{-1/2})$. Hence
\begin{align*}
T(n;N) =& \frac{1}{\Gamma(1/4)^2 \sqrt{n}} \int_0^1 t^{-3/4} (1-t)^{-3/4} \mathbb{P}(\alpha_1 \leq \frac{s}{t}) \mathbb{P}(\alpha_1 \leq \frac{s}{1-t}) \, dt + O(\epsilon^{1/4} n^{-1/2}) + o(n^{-1/2}) \\
=& \frac{1}{\sqrt{\pi n}} \Big(\int_0^1 \mathbb{E} \mathbf{1}( 1 - \frac{s}{\alpha_1} \leq t \leq \frac{s}{\alpha_1'} ) \sqrt{\pi} \Gamma(1/4)^{-2} t^{-3/4} (1-t)^{-3/4}\, dt + O(\epsilon^{1/4})+o(1)\Big) \\
=& \frac{1}{\sqrt{\pi n}} \Big(\mathbb{P}\Big(1- \frac{s}{\alpha_1} \leq Y \leq \frac{s}{\alpha_1'}\Big) + O(\epsilon^{1/4})+ o(1)\Big),
\end{align*}
where $\alpha_1'$ is an independent copy of $\alpha_1$ and $Y \sim \mathrm{Beta}(1/4,1/4)$. As $\epsilon$ is arbitrary this establishes the claim.
\end{proof}

\subsection{From $\mathbb{F}_q[T]$ to $\mathbb{Z}$: sums of squares}\label{sec:making_integer_conjectures}
Theorem \ref{main_thm_limit} suggests a conjecture for the integers regarding the number of elements of $S$ that lie in a short interval. Naively one might think it will suggest a conjecture regarding the quantity
\begin{equation}
\label{prob_variance_prelim}
\frac{1}{X}\int_X^{2X}\Big( \sum_{x \leq n \leq x+H} b(n) - M_{X,H}\Big)^2\, dx,
\end{equation}
where $H = X^{\delta}$ with $\delta \in (0,1)$ and
\begin{equation}
\label{prob_mean_prelim}
M_{X,H} = \frac{1}{X} \int_{X}^{2X} \sum_{x < n \leq x+H} b(n)\, dx \sim K \frac{H}{\sqrt{\log X}}.
\end{equation}
Here \eqref{prob_variance_prelim} is the \emph{probabilistic variance} of $\sum_{x \leq n \leq x+H} b(n)$ and \eqref{prob_mean_prelim} is the \emph{probabilistic mean}. This is not exactly the right quantity to look at, owing to the fact that $b(n)$ on average behaves like $1/\sqrt{\log n}$, and the slow change of this function means that the variance in \eqref{prob_variance_prelim} will be much larger than we would like. Indeed, even the probabilistic variance of $\sum_{x\leq n \leq x+H} 1/\sqrt{\log n}$ is quite large owing to this change; the probabilistic variance of this sum is
$$
\frac{1}{X} \int_X^{2X} \Big( \sum_{x \leq n \leq x+H} \frac{1}{\sqrt{\log n}} - \frac{1}{X} \int_X^{2X} \sum_{t \leq n \leq t+H} \frac{1}{\sqrt{\log n}}\, dt\Big)^2\, dx,
$$
and with a little work one may see that this is at least of order $H^2/(\log X)^3$.

Thus instead of \eqref{prob_variance_prelim}, we consider a variant in which $M_{X,H}$ has been replaced by a better approximation to $\sum_{x< n \leq x+H} b(n)$ which changes with $x$; this approximation is given in terms of an integral of $L$-functions.

Define the function $F(s)$ for $\Re s > 1$ by
\begin{equation}
\label{eq:F_def}
F(s) = \sum_{n=1}^\infty \frac{b(n)}{n^s}.
\end{equation}
Using the fact that $n$ is an element of $S$ if and only if $n$ can be written in the form $2^\alpha \mu \nu^2$, for $\mu$ a product of primes congruent to $1$ modulo $4$ and $\nu$ a product of primes congruent to $3$ modulo $4$, it may be seen that for $\Re s > 1$,
\begin{align}
\label{eq:F_prod}
F(s) =& \frac{1}{1-2^{-s}} \prod_{q \equiv 1 \bmod 4} \frac{1}{1-q^{-s}} \prod_{r \equiv 3 \bmod 4} \frac{1}{1-r^{-2s}} \\
=&\Big( \frac{\zeta(s) L(s,\chi_4)}{1-2^{-s}}\Big)^{1/2} \prod_{k=1}^\infty \Big( \frac{\zeta(2^k s)}{L(2^k s, \chi_4)} (1-2^{-2^k s})\Big)^{1/2^{k+1}},
\end{align}
where $\chi_4$ is the non-principal character modulo $4$. The first Euler product here dates at least back to Landau \cite{land}, while the second factorization has in effect been derived many times (see e.g. \cite{shanks1964the,flajolet1996zeta}).

The second representation allows one to analytically continue $F(s)$ to the cut disc $\mathcal{E} = \{s: |s-1| < 1/2\} \setminus \{s: \Im s = 0, \Re s \leq 1\}$: note that in this region, because neither $\zeta(s)$ nor $L(s,\chi)$ have low-lying zeros inside of it (see \cite{lmfdb} for a list of zeros), we can write 
\begin{equation}
\label{eq:F_factor}
F(s) = (s-1)^{-1/2} f(s),
\end{equation}
where $f(s)$ is an analytic function and where the principal branch of the function $(s-1)^{-1/2}$ is taken.

Assuming the Riemann Hypothesis for $\zeta(s)$ and $L(s,\chi_4)$, we show in Theorem \ref{thm:B_approx} that for any $\epsilon > 0$,
$$
B(x) = \overline{B}(x) + O_{\epsilon}(x^{1/2+\epsilon}), \quad \textrm{where}\quad \overline{B}(x)= \frac{1}{\pi} \int_{1/2}^1 \frac{x^s}{(1-s)^{1/2} s} f(s) \, ds.
$$
Thus we approximate $B(x+H)-B(x)$ (the number of elements of $S$ in a short interval $(x,x+H]$) by
\begin{equation}
I(x;H):= \overline{B}(x+H)-\overline{B}(x).
\end{equation}

We will consider variance defined in the following sense:
\begin{equation}
V_b(X;H):= \frac{1}{X}\int_X^{2X} (B(x+H)-B(x) - I(x;H))^2\, dx,
\end{equation}
Ramachandra \cite{ramachandra1976some} investigated a quantity equivalent to this one and showed that there is some cancellation over the trivial bound of $H^2/\log X$; namely
\begin{equation}
V_b(X;H) = O(H^2 \exp[-(\log X)^{1/6}]),
\end{equation}
for $H > X^{1/6+\epsilon}$. Under density hypotheses for the zeros of $\zeta(s)$ and $L(s,\chi_4)$ (see \cite[Eq. (6)]{ramachandra1976some}) this is improved to the more complete range $H > X^\epsilon$.

Motivated by Theorem \ref{main_thm_Variance}, we believe that

\begin{conj}
\label{b_variance_conj}
Fix $\delta \in (0,1)$. As $X\rightarrow\infty$ with $H = X^\delta$, we have
\begin{equation}
V_b(X;H) = \Big(K\, G(1-\delta) + o(1)\Big)\frac{H}{\sqrt{\log X}},
\end{equation}
for $K$ as in \eqref{K_def} and $G(s)$ as in \eqref{G_def}.
\end{conj}

Theorem \ref{G_positivity} thus suggests the perhaps more tractable conjecture that for $H = X^\delta$ with fixed $\delta \in (0,1),$
\begin{equation}
\frac{H}{\sqrt{\log X}} \ll V_b(X;H) \ll \frac{H}{\sqrt{\log X}},
\end{equation}
with implicit constants depending on $\delta$.

Returning to Figure \ref{fig:data_and_prediction_short_interval}, there for $X = 10^8$ we have plotted the numerical value of the points $(\delta, \frac{V_b(X;H)}{H/\sqrt{\log X}})$ for $\delta = \log(H)/\log X$ for various primes $H$, and in comparison have also plotted the curve $(\delta, K \, G(1-\delta))$. 

Replacing random short intervals with random sparse arithmetic progressions, it is also reasonable to believe in a variant of Conjecture \ref{b_variance_conj}. Some of the analytic difficulties which arise in defining $V_b(X;H)$ vanish in this context. For $X$ and $q$ positive integers, define
\begin{equation}\label{eq:vbxq}
\mathbb{V}_b(X,q) := \frac{1}{\phi(q)} \sum_{\substack{1\leq a \leq q  \\ (a,q)=1}} \Big( \sum_{\substack{n \equiv a \bmod q \\ n \leq X}} b(n) - \frac{1}{\phi(q)} \sum_{\substack{(n,q)=1 \\ n\leq X}} b(n)\Big)^2.
\end{equation}
Note that in contrast to the definition of $V_b(X;H)$, the quantity $\mathbb{V}_b(X,q)$ genuinely is the probabilistic variance of counts of elements of the set $S$ that lie in a random arithmetic progression; we are able to consider the probabilistic variance because the density of the set $S$ does not change as we vary over arithmetic progressions modulo the same number. One may think in this set up of $X/q$, roughly the number of elements in each such arithmetic progression, as playing the role of $H$ above. For the sake of simplicity we make a conjecture only for prime moduli.
\begin{conj}
\label{b_variance_conj_q}
Fix $\delta \in (0,1)$. As $X\rightarrow\infty$ choose primes $p$ such that $X/p = X^{\delta+o(1)}$. Then
\begin{equation}
\mathbb{V}_b(X,p) = \Big(K\, G(1-\delta) + o(1)\Big) \frac{X/p}{\sqrt{\log X}},
\end{equation}
for $K$ as in \eqref{K_def} and $G(s)$ as in \eqref{G_def}.
\end{conj}

One may likewise conjecture that for $X/p = X^{\delta + o(1)}$ and fixed $\delta \in (0,1)$,
$$
\frac{X/p}{\sqrt{\log X}} \ll \mathbb{V}_b(X,p) \ll \frac{X/p}{\sqrt{\log X}}.
$$
Very recently an averaged version of the lower bound has been established for $\delta \in (0,1/2)$ by Mastrostefano \cite{mastrostefano2020lower}.

\begin{figure}[h]
  \includegraphics[width=\linewidth]{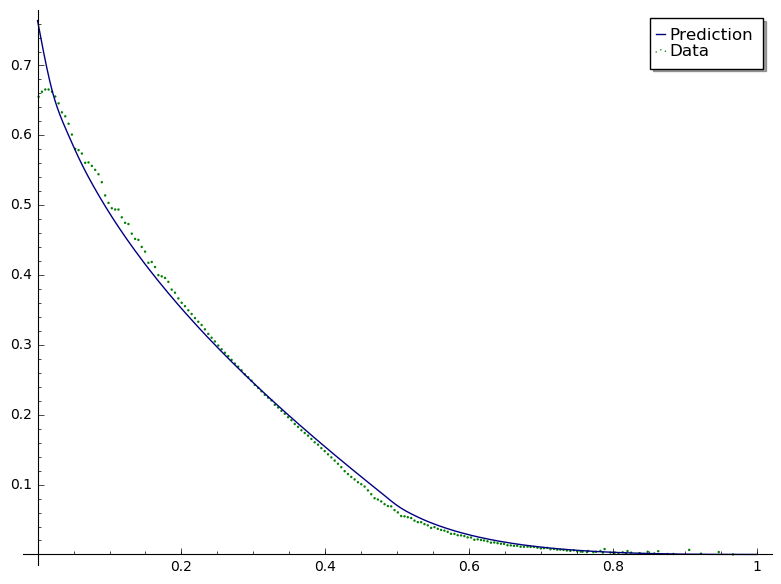} \caption{Numerically produced data compared to the prediction of Conjecture \ref{b_variance_conj_q} for variance in arithmetic progressions. For $X = 9\cdot 10^8$ and a selection of primes $p$, set $\delta \log(X/p)/\log(X)$; we plot the points $(\delta, \mathbb{V}_b(X,p)/((X/p)/\sqrt{\log X}))$ under the label \textbf{data}, and the curve $(\delta, K\, G(1-\delta))$ under \textbf{prediction}.}
\label{fig:data_and_prediction_AP}
\end{figure}

\begin{remark}
The reader may discern a bump in the data near $\delta=0$ in Figure~\ref{fig:data_and_prediction_AP}. This is explained by lower order terms which become negligible in the $X \to \infty$ limit. See Appendix \ref{sec:appendix large p} for a discussion.
\end{remark}

\subsection{From $\mathbb{F}_q[T]$ to $\mathbb{Z}$: divisor sums}\label{sec:making_integer_conjectures2}

Theorem \ref{d_z_limit} likewise suggests a conjecture over the integers for the arithmetic functions $d_z(n)$. For $z > 0$, we approximate
\begin{equation}
\label{eq:D_z_def}
D_z(x):= \sum_{n\leq x} d_z(n)
\end{equation}
by the function
\begin{equation}
\label{eq:D_z_approx}
\overline{D}_z(x):= \frac{1}{2\pi i} \int_{\mathcal{C}_{3/4}} \frac{x^s}{s} \zeta(s)^z\, ds,
\end{equation}
where $\mathcal{C}_{3/4}$ is a contour tracing out the cut circle $\{s:\, |s-1| = 3/4,\, s\neq 1/4\}$ in the counterclockwise direction, and the continuation of $\zeta(s)^z$ which is real for real values of $s$ is taken on this contour. (Here we are recalling \eqref{divisor_generating} that $\zeta(s)^z$ is the Dirichlet series for $d_z(n)$.) When $z$ is an integer the integral reduces to just a residue at $s=1$ and becomes $x$ multiplied by a polynomial in $\log x$, but otherwise $\overline{D}_z(x)$ is a more complicated expression, see \cite[Ch. 14.6]{ivic1985riemann} for an asymptotic expansion. We have $D_z(x) = \overline{D}_z(x) + O_{\epsilon}(x^{1/2+\epsilon})$ on the assumption of the Riemann Hypothesis; see Appendix \ref{sec:B(x)}.

Define
\begin{equation}
\Delta_z(x):= D_z(x) - \overline{D}_z(x),
\end{equation}
\begin{equation}
\Delta_z(x;H):= \Delta_z(x+H) - \Delta_z(x),
\end{equation}
and consider
\begin{equation}
V_{d_z}(X;H):= \frac{1}{X} \int_X^{2X} \Delta_z(x;H)^2\, dx.
\end{equation}

Where $z$ is an integer this quantity was investigated in \cite{keating2018sums}, who made a conjecture \cite[Conjecture 1.1]{keating2018sums} regarding its asymptotic value. On the basis of Theorem \ref{d_z_limit} we believe that conjecture generalizes in the following way:
\begin{conj}
\label{d_z_variance_conj}
Fix $z > 0$. For $\delta \in (0,1)$ fixed, and $H = X^\delta$,
\begin{equation}
V_{d_z}(X;H) = a_z \Big(\frac{\mathbb{P}(\alpha_1^{(z)} \leq 1-\delta)}{\Gamma(z^2)} + o(1)\Big) H (\log X)^{z^2-1},
\end{equation}
as $X\rightarrow\infty$, where
\begin{equation}
a_z = \lim_{s\rightarrow 1^+} (s-1)^{z^2} \sum_{n=1}^\infty \frac{d_z(n)^2}{n^s}.
\end{equation}
\end{conj}

Likewise we may define
$$
\mathbb{V}_{d_z}(X,p) = \frac{1}{\phi(q)} \sum_{\substack{1\leq a \leq q  \\ (a,q)=1}} \Big( \sum_{\substack{n \equiv a \bmod q \\ n \leq X}} d_z(n) - \frac{1}{\phi(q)} \sum_{\substack{(n,q)=1 \\ n\leq X}} d_z(n)\Big)^2.
$$

\begin{conj}
\label{d_z_variance_conj_q}
Fix $z > 0$ and fix $\delta \in (0,1)$. As $X\rightarrow\infty$ choose primes $p$ such that $X/p = X^{\delta+o(1)}$. Then
\begin{equation}
\mathbb{V}_{d_z}(X,p) = a_z \Big(\frac{\mathbb{P}(\alpha_1^{(z)} \leq 1-\delta)}{\Gamma(z^2)} + o(1)\Big) (X/p) (\log X)^{z^2-1}
\end{equation}
as $X\rightarrow\infty$.
\end{conj}

For work towards these conjectures for integer $z$, see for example the recent works \cite{ivic1985riemann, lester2016variance, rodgers2018variance, harper2017lower, de2020major, bettin2020averages}.

\begin{appendix}

\section{More on $z$-measures: scaling limits, positivity, and open problems}\label{sec:Further}

\subsection{Scaling limits: on Theorem \ref{thm:z_scalinglimit}}

We now turn to a proof of Theorem \ref{thm:z_scalinglimit}. The large part of this theorem has explicitly appeared in the literature before: for $z \in \mathbb{C}\setminus \mathbb{Z}_{\leq 0}$ and $\lambda \vdash n$ chosen according to parameters $z$ and $z'=\overline{z}$, the fact that
$$
\frac{\lambda_1}{n} \rightarrow \alpha_1^{(z)}
$$
in distribution is a consequence of \cite[Thm. 1.6]{borodin2005z}.

On the other hand this does not guarantee the continuity of the function $F_z(x)$ in \eqref{CDF_defn}. We establish continuity by breaking into two cases. Having established the continuity of $F_z(x)$, this will imply \eqref{z_limiting} by the Portmanteau theorem. (We do not establish continuity for $z =1$ so this does not work in that case, but for $z=1$ the convergence in \eqref{z_limiting} is obvious -- see Remark \ref{z_remark}.)

The separate cases we consider are $z \in \mathbb{Z}_{\geq 2}$ and $z \in \mathbb{C}\setminus \mathbb{Z}$. Clearly together these cases cover $z \in \mathbb{C}\setminus \mathbb{Z}_{\leq 1}$. In the terminology of \cite{borodin2005z}, $z \in \mathbb{C}\setminus \mathbb{Z}$ induces $z$-measures in the \emph{principal series}, while $z \in \mathbb{Z}_{\ge 2}$ induces $z$-measures in the \emph{degenerate series}.

\subsubsection{$z \in \mathbb{Z}_{\geq 2}$ and relation to the $\gamma_k$ integral}

We consider elements of $\mathbb{Z}_{\geq 2}$ first. We write $k\in \mathbb{Z}_{\geq 2}$ instead of $z$ to emphasize that these are integers, and our goal in this section is to prove the continuity of $F_k(x)$ for all $x\in \mathbb{R}$. 

Our main tool will be to show that the limiting cumulative distribution functions $F_k(c)$ are related to the integral defined by \cite[(1.12)]{keating2018sums}:
\begin{equation}
\label{gamma_integral}
\gamma_k(c) := \frac{1}{k! G(k+1)^2} \int_{[0,1]^k} \delta_c(w_1+\ldots + w_k) \Delta(w)^2\, d^k w,
\end{equation}
where $G$ is the Barnes $G$-function with $G(k+1) = \prod_{j=1}^{k-1} j!$ for integer $k$, $\delta_c(w) := \delta(w-c)$ is the delta distribution translated by $c$, and $\Delta(w):= \prod_{i < j} (w_i - w_j)$ is the Vandermonde determinant.
\begin{proposition}
\label{prop:integral_rep}
For $k \in \mathbb{Z}_{\geq 2}$,
\begin{equation}
\label{F_k_to_gamma}
F_k(s) = (k^2-1)! s^{k^2-1} \gamma_k(s^{-1}), \quad \textrm{for}\; s\in (0,1).
\end{equation}
\end{proposition}
Since we have defined the cumulative distribution function $F_z(s)$ in terms of the limit \eqref{z_limiting}, the content of \eqref{F_k_to_gamma} is that
\begin{equation}
\label{F_k_to_gamma_2}
\lim_{n\rightarrow\infty} \mathbb{P}\Big(\frac{\lambda_1}{n}\leq s\Big) = (k^2-1)! s^{k^2-1} \gamma_k(s^{-1}), \quad \textrm{for}\; s\in(0,1).
\end{equation}
From the definition \eqref{gamma_integral} it is plain that $\lim_{s\rightarrow 0^+} s^{k^2-1} \gamma_k(s^{-1}) = 0$ and from \cite[Sec. 4.4.3]{keating2018sums} it is plain $\lim_{s\rightarrow 1^{-}}s^{k^2-1} \gamma_k(s^{-1}) = 1$, and moreover it follows directly from the definition \eqref{gamma_integral} that $\gamma_k(c)$ is continuous. Because the limiting cumulative distribution function obviously has $F_z(x) = 0$ or $1$ if $x \leq 0$ or $x \geq 1$ respectively, Proposition \ref{prop:integral_rep} therefore implies the continuity of $F_z(x)$ for all $x \in \mathbb{R}$.

It remains then to verify Prop. \ref{prop:integral_rep}. Note that from Theorem \ref{RMT_to_zmeasures}, where $\lambda \vdash n$ is drawn according to $z$-measure with parameters $k,k$,
\begin{equation}
\int_{U(N)} A_{n,(k)}(g) A_{n,(k)}(g^{-1})\, dg = \frac{(k^2)_n}{n!} \mathbb{P}(\lambda_1 \leq N),
\end{equation}
while using the notation of (1.27) in \cite{keating2018sums}, the integral here is equal to
\begin{equation}
I_k(n;N).
\end{equation}
On the other hand, by \cite[Thm. 1.5]{keating2018sums},
\begin{equation}
I_k(n;N) = \gamma_k(n/N) N^{k^2-1} + O_k(N^{k^2-2}).
\end{equation}
(This is stated there for $0\leq n \leq kN$, but since the left hand side and the main term of the right hand side vanish outside this range it remains true for all $n, N$.) Furthermore $(k^2)_n/n! = n^{k^2-1}/(k^2-1)! + O(n^{k^2-2})$. Thus
\begin{equation}
\mathbb{P}\Big(\frac{\lambda_1}{n} \leq \frac{N}{n}\Big) = (k^2-1)! (N/n)^{k^2-1} \gamma_k(n/N) + O(1/\min(n,N)),
\end{equation}
and if $n\rightarrow\infty$ with $N = \lfloor sn \rfloor$, this establishes \eqref{F_k_to_gamma_2}.

\subsubsection{$z \in \mathbb{C}\setminus \mathbb{Z}$}
\label{C_less_Z}

Establishing continuity of the limiting cumulative distribution function for $z \in \mathbb{C}\setminus \mathbb{Z}$ requires different tools. We have noted already that we know there exists a random variable $\alpha_1^{(z)}$ such that $\lambda_1/n \rightarrow \alpha_1^{(z)}$ in distribution. It was observed in \cite{olshanski1998point} that $\alpha_1^{(z)}$ is best studied not in isolation but as the largest element of a stochastic point process with configurations
\begin{equation}
\label{omega_points}
\omega = (\alpha_1^{(z)}, \alpha_2^{(z)},...; - \beta_1^{(z)}, -\beta_2^{(z)},...)
\end{equation}
such that the points $\alpha,\beta$ lie on the Thoma simplex, $\Omega = \{\alpha_1 \geq \alpha_2 \geq ... \geq 0; \beta_1 \geq \beta_2 \geq ... \geq 0 :\; \sum_i \alpha_i + \beta_i \leq 1\}$. This perspective was further pursued in \cite{borodin1998point}. Background we will require about point processes can be found in \cite[Sec. 1]{soshnikov2000determinantal}.

In \cite{borodin1998point}, correlation functions for this point process are explicitly computed. We explain the results from this paper that we will use. In the first place, for $f$ a continuous function supported in the cube $[-1,1]^n$, we have
\begin{equation}
\mathbb{E} \sum_{\substack{j_1,...,j_n \\ \textrm{distinct}}} f(\omega_{j_1}, ..., \omega_{j_n}) = \int_{[-1,1]^n} f(x) \mu_n(dx_1, ..., dx_n),
\end{equation}
where the sum on the left hand side is over all collections of $n$ distinct indices of the configuration $\omega$, and where $\mu_n$ is a measure such that (see \cite[Thm. 2.4.1]{borodin1998point}) in the region $x_1, ..., x_n > 0$ and $\prod_{i\leq j} (x_i-x_j) \neq 0$ and $\sum x_i < 1$,
\begin{equation}
\label{mu_to_corr}
\mu_n(dx_1, ..., dx_n) = \rho_n(x_1,...,x_n) dx_1\cdots dx_n,
\end{equation}
for a function $\rho_n$ continuous in this region. Moreover (see \cite[Thm. 6.1]{olshanski1998point}), the measure $\mu_n$ is supported on the set $\{x:\; |x_1| + \cdots + |x_n| \leq 1\}$, and (see the beginning of the proof of \cite[Thm. 3.3.1]{borodin1998point}) $\mu_n$ is non-singular on the set $\{x:\; \sum |x_i| =1\}$, and (see \cite[Sec 2.5]{borodin1998point}, and also Remarks 2.2.2 and 2.5.3) $\mu_n$ is non-singular on the set $\prod_{i< j} (x_i-x_j) \neq 0$.

Thus taking all these facts we need together, we have that
\begin{equation}
\mathbb{E} \sum_{\substack{j_1,...,j_n \\ \textrm{distinct}}} f(\omega_{j_1}, ..., \omega_{j_n}) = \int_{[-1,1]^n} f(x) \rho_n(x_1,...,x_n) dx_1\cdots dx_n,
\end{equation}
where the correlation functions $\rho_n$ are supported on $x$ with $\sum |x_i| \leq 1$.

We now note that for $s \in (0,1)$, we have
\begin{equation}
\label{to_gap_prob}
\mathbb{P}(\alpha_1^{(z)} \leq s) = \mathbb{P}(\#_{(s,1]} = 0),
\end{equation}
where $\#_{(s,1]}$ is the random variable describing the number of points of the configuration $\omega$ lying in the interval $(s,1]$. By a well-known expression (see \cite[Prop 2.4, (2.22)]{johannson2005random}),
\begin{equation}
\label{gap_to_correlations}
\mathbb{P}(\#_{(s,1]} = 0) = 1+ \sum_{n=1}^\infty \frac{(-1)^n}{n!} \int_{(s,1]^n} \rho_n(x_1,...,x_n)\, dx_1 dx_2\cdots dx_n.
\end{equation}
The sum in \eqref{gap_to_correlations} converges and in fact for each $s\in(0,1)$ has only finitely many non-zero terms. To see this note that we have (\cite[(1.5)]{soshnikov2000determinantal}) 
\begin{equation}
\int_{(s,1]^n} \rho_n = \mathbb{E}\,\#_{(s,1]}(\#_{(s,1]}-1)\cdots (\#_{(s,1]}-(n-1)).
\end{equation} 
As $\sum \alpha_i + \beta_i \leq 1$ one sees that $\#_{(s,1]} \leq \lfloor 1/s \rfloor$, from which it follows that $\int_{(s,1]^n} \rho_n = 0$ for $n \geq \lfloor 1/s \rfloor$.

Thus from \eqref{to_gap_prob} and \eqref{gap_to_correlations} it will follow that $F_z(s)$ is continuous for $s\in (0,1)$ if each summand in \eqref{gap_to_correlations} is continuous in $s$. But this follows from the function $\rho_n$ being continuous (or indeed just a measurable function).

Thus we have shown that $F_z(s)$ is continuous for $s \in (0,1)$, and since obviously for the limiting cumulative distribution function we have $F_z(s) = 0$ or $1$ if $s\leq 0$ or $s\geq 1$ respectively, we need only show that $\lim_{s\rightarrow 0^+} F_z(s) = 0$ and $\lim_{s\rightarrow 1^-} F_z(s) = 1.$ The latter follows from the same continuity argument as above, while the former if false would imply that $\#_{(0,1]} = 0$ occurs with positive probability (by intersecting the nested events $\#_{(s,1]} = 0$). But $\#_{(0,1]} = 0$ implies for point configurations that $\alpha_1 = \alpha_2 = ... = 0$, and if this occurred with positive probability it would contradict the fact that the point processes we are considering are simple (that is, it will not happen that multiple points of a configuration coincide at the same location; see \cite[Sec. 2.5]{borodin1998point}). This completes the proof of continuity for $z \in \mathbb{C}\setminus \mathbb{Z}$, and therefore of Theorem \ref{thm:z_scalinglimit}.

\begin{figure}
\label{fig:P_z_60}
\hspace{-65pt}
\begin{subfigure}[b]{.65\textwidth}
  \includegraphics[width=\textwidth]{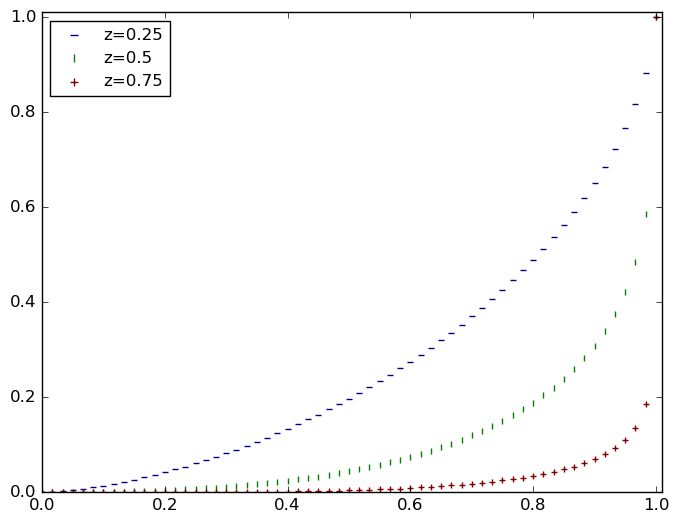}
  \caption{}
  \label{fig:P_z_60_lessthan1} 
\end{subfigure}
~
\begin{subfigure}[b]{.65\textwidth}
  \includegraphics[width=\textwidth]{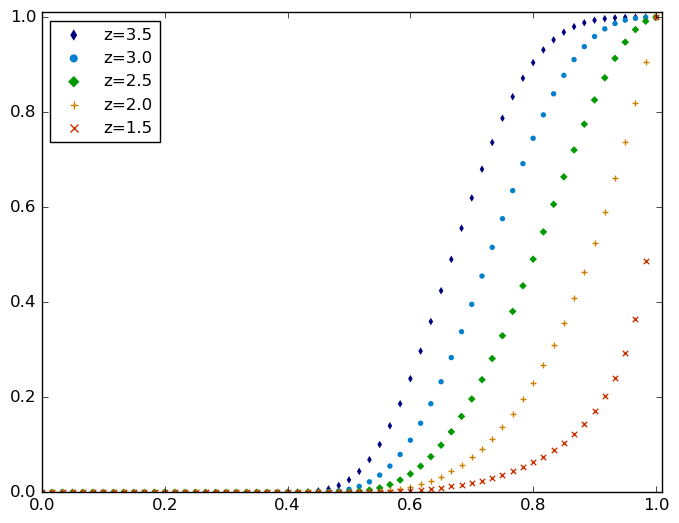}
  \caption{}
  \label{fig:P_z_60_biggerthan1}
\end{subfigure}

\centering
\begin{subfigure}[b]{.7\textwidth}
  \includegraphics[width=\linewidth]{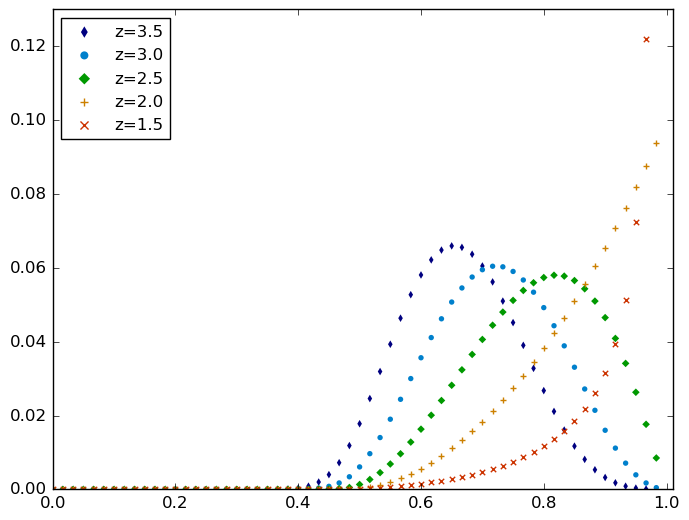}
  \caption{}
  \label{fig:P_z_60_biggerthan1_deriv}
\end{subfigure}

\caption{\textbf{(a)} A finite approximation of $F_z(s)$ for $z=0.25$, $0.5$, $0.75$. \textbf{(b)} A finite approximation for $F_z(s)$ for $z=1.5$, $2.0$, $2.5$, $3.0$, $3.5$. In both graphs, the finite approximations are the measures $\mathbb{P}_z^{(n)}(\lambda_1 \leq sn)$ for $n = 60$ with data points taken at $sn$ an integer. \textbf{(c)} A finite approximation of $F^\prime_z(s)$, obtained from the discrete derivative of the graphs in (b).
}
\end{figure}

\begin{figure}
\centering

	\begin{subfigure}[b]{\textwidth}
\centering
\includegraphics[width=.8\textwidth]{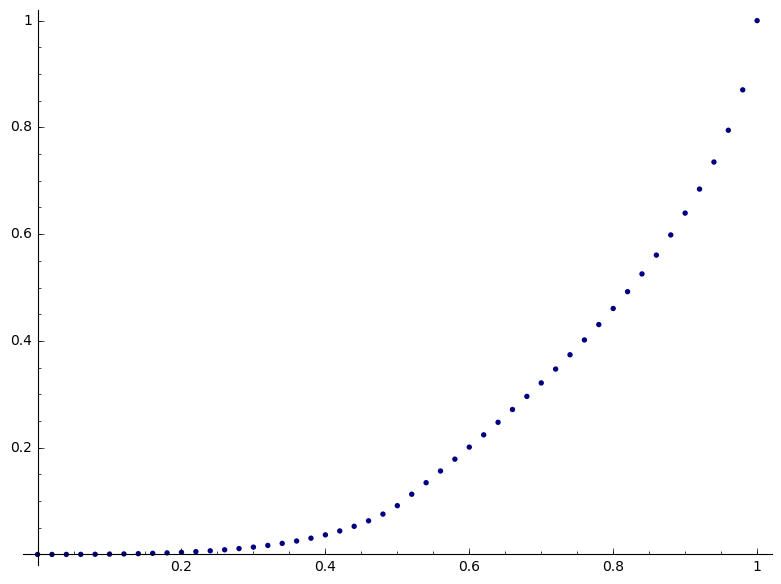}
\caption{}
\label{fig:G_n_50}
\end{subfigure}

\begin{subfigure}[b]{.8\textwidth}
  \includegraphics[width=\linewidth]{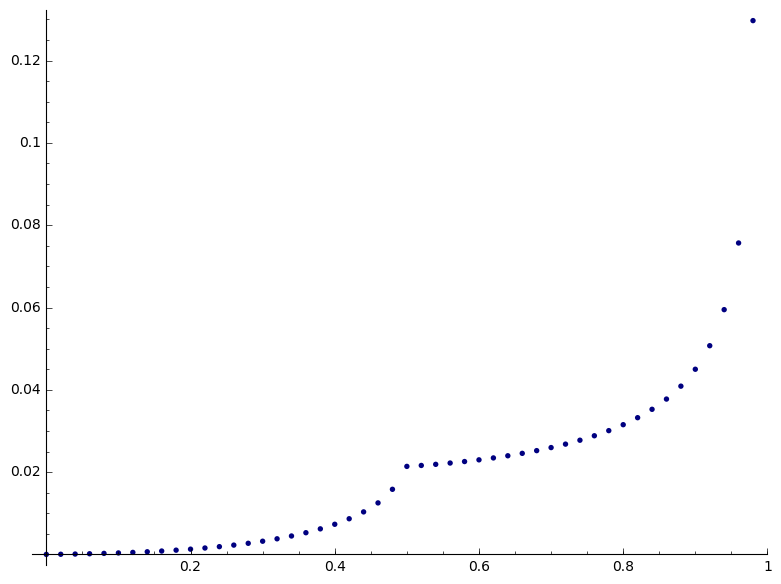}
  \caption{}
  \label{fig:G_n_50_deriv}
\end{subfigure}

\caption{\textbf{(a)} A finite approximation of $G(s)$, obtained from $T(n;N)/T(n;n)$ for $n=50$. \textbf{(b)} A finite approximation of $G'(s)$, obtained from the discrete derivative of the graph in (a).}
\end{figure}

\subsection{Positivity: on Theorems \ref{G_positivity} and \ref{thm:F_z_positivity}}

We now turn to the claims made in Section \ref{subsec:ffanalogue} about the positivity of the functions $F_z(s)$ and $G(s)$. We treat $F_z(s)$ first.

\begin{proof}[Proof of Theorem \ref{thm:F_z_positivity}]
For integer $k \ge 2$, an examination of \eqref{gamma_integral} shows that $\gamma_k(c)$ is supported on the interval $c \in [0,k]$ and non-zero on the interior of this interval. Hence by Proposition \ref{prop:integral_rep}, $F_k(s)$ vanishes for $s\in[0,1/k]$ and is positive for $s>1/k$ as claimed.

For non-integer $z > 0$, this is a direct result of the main Theorem of \cite{olshanski2018topological}; this Theorem is that for any non-degenerate $z$-measure (of which non-integer $z > 0$ is a special case), the topological support of the $z$-measure is the whole Thoma simplex $\Omega$. 
\end{proof}

\begin{remark}
While Theorem \ref{thm:F_z_positivity} shows that $F_z(s)>0$ for all positive $s$ if $z$ is not an integer, this function will nonetheless take extremely small values when $z > 1$. See Figure \ref{fig:P_z_60_biggerthan1}. (We do not know of any quantitative estimates in the literature.)
\end{remark}

The positivity of $G(s)$ is a corollary of that of $F_z(s)$.

\begin{proof}[Proof of Theorem \ref{G_positivity}]
Recall $G(s) = \mathbb{P}(1-s/\alpha_1 \leq Y \leq s/\alpha_1')$. Since $\alpha_1$ and $\alpha_1'$ can be made arbitrarily small with positive probability, for any $s$ one can make $1-s/\alpha_1\leq 1/4$ and $s/\alpha_1'\geq 3/4$ with positive probability. Since for $Y \sim \mathrm{Beta}(1/4,1/4)$ we have $1/4 \leq Y \leq 3/4$ with positive probability, the claim that $G(s) > 0$ follows.
\end{proof}

\subsection{Questions about derivatives}
\label{further_derivatives}
The existence and continuity of the derivatives $F_z'(s) = \tfrac{d}{ds} F_z(s)$ for $s \in (0,1)$ is implied by \eqref{F_k_to_gamma} and \eqref{to_gap_prob}--\eqref{gap_to_correlations}. Figure \ref{fig:P_z_60_biggerthan1_deriv} numerically plots these and suggests the following:

\begin{conj}
\label{deriv_behavior}
For $z \in (0,1)\cup(1,2]$, $F'_z$ is a monotone increasing function, while for $z>2$, $F'_z$ is unimodal with a maximum attained for some $\delta \in (0,1)$.
\end{conj}

It is possible with explicit computation to see that the derivative $G'(s)$ is well defined and continuous for $s \in(0,1)$, though we leave details to the reader. The function $G(s)$ appears to exhibit a phase change at $s=1/2$, and this becomes most apparent in a graph of $G'(s)$; in this graph there seems to be a kink at $s=1/2$ (see Figure \ref{fig:G_n_50_deriv}). This phase change can be understood to be a consequence of the $z=1/2$ case of Conjecture \ref{deriv_behavior}. We outline how so, but leave details to the reader: by an explicit computation, using the continuity of $F'_{1/2}(s)$ for $s \in (0,1)$, the second derivative $G''(s)$ can be seen to be well defined at all points except $s=1/2$. Furthermore, at $s=1/2$, one can show from explicit computation and Conjecture \ref{deriv_behavior} that $\lim_{s\rightarrow 1/2^-} G''(s) \neq \lim_{s\rightarrow 1/2^-} G''(s)$. The key input from the Conjecture is the implication that $\lim_{s\rightarrow 1^-}F'_{1/2}(s) \neq 0$; one also needs to use the fact that for $s > 1$, $F_z'(s) = 0$, which is evident from the definition.

Likely other phase changes in the function $G(s)$ exist as $s$ varies and become visible in higher derivatives (one might expect at $s = 1, 1/2, 1/3, 1/4,...$). Conjectures \ref{b_variance_conj} and \ref{b_variance_conj_q} suggest that phase changes in $G(s)$ should have an arithmetic interpretation, and understanding this remains an interesting problem.

\subsection{Questions about log-concavity and a gaussian limit}
\label{further_logconcave}

In \cite{basor2018some} it is shown (as a consequence of Theorem 3.1 there) that the functions $\gamma_k(c)$ have a gaussian limit shape in the sense that
\begin{equation}
\lim_{k\rightarrow\infty} \frac{G(2k+1)}{G(k+1)^2} \gamma_k(c) = \sqrt{\frac{8}{\pi}} e^{-8t^2}, \quad \textrm{for}\; c = \tfrac{k}{2} + t.
\end{equation}
(c.f. \cite{lambert2019stein} for related but different results in random matrix theory.)

One may also prove that the function $\gamma_k(c)$ is log-concave for all $k$. This observation seems to be new; a sketch of a proof is simply as follows: use the integral representation \cite[(4.22)]{keating2018sums} and the fact that marginals preserve log-concavity (see \cite[Theorem 3.3]{saumard2014log}).

One may sensibly ask the same questions when \eqref{F_k_to_gamma} is used to replace integer $k$ by a continuous parameter. That is, is it true that
\begin{equation}
\lim_{z\rightarrow\infty} \frac{G(2z+1)}{\Gamma(z^2) G(z+1)^2} c^{z^2-1} F_z(1/c) = \sqrt{\frac{8}{\pi}} e^{-8t^2}, \quad \textrm{for}\; c = \frac{z}{2}+t?
\end{equation}
Furthermore, is $c^{z^2-1} F_z(1/c)$ log-concave in $c$? We do not know the answer to these questions.

\section{Approximating $B(x)$}
\label{sec:B(x)}

Recall $B(x)$ counts the number of natural numbers less than or equal to $x$ which can be represented as sums of two squares, and the corresponding Dirichlet series is $F(s)$, defined by \eqref{eq:F_def}. $F(s)$ satisfies the factorization \eqref{eq:F_factor} for $\Re s > 1$, and assuming the Riemann Hypothesis for $\zeta(s)$ and $L(s,\chi_4)$, $F(s)$ has an analytic continuation to the cut half-plane $\{s: \Re s > 1/2,\, s\notin (1/2,1]\}$, and satisfies $F(s) = (s-1)^{-1/2} f(s)$ for a function $f(s)$ analytic in $\Re s > 1/2$, where the principal branch of the square root function is taken. Our purpose in this appendix is to prove the following folklore result:

\begin{thm}
\label{thm:B_approx}
On the assumption of the Riemann Hypothesis for $\zeta(s)$ and $L(s,\chi_4)$, for any $\epsilon > 0$,
$$
B(x) = \frac{1}{\pi} \int_{1/2}^1 \frac{x^s}{(1-s)^{1/2} s} f(s)\, ds + O_\epsilon(x^{1/2+\epsilon}).
$$
\end{thm}

\begin{proof}
We have by Perron's formula (see \cite[Cor. 5.3]{montgomery07multiplicative}), for $T = x^{100}$,
$$
B(x) = \frac{1}{2\pi i} \int_{2 - i T}^{2+i T} \frac{x^s}{s} F(s)\, ds + O(1).
$$
For arbitrary $\epsilon > 0$, let $\sigma = 1/2+\epsilon$, and let $\mathcal{K}_\delta$ be a contour from $\sigma-i\delta$ to $1+\delta - i\delta$ to $1 +\delta + i \delta$ to $\sigma + i \delta$ for $\delta > 0$. On the Riemann Hypothesis the contour from $2-iT$ to $2+iT$ may be shifted to a contour from $2 - i T$ to $\sigma- iT$ to $\sigma - i\delta$, followed by $\mathcal{K}_\delta$, followed by a contour from $\sigma+i\delta$ to $\sigma+iT$ to $2 + iT$. The Lindel\"of estimates $\zeta(s), L(s,\chi_4) = O_\epsilon(|s|^\epsilon)$ for $\Re s \geq 1/2,$ $|s-1|\geq 1/10$ can be used to bound those contours other than $\mathcal{K}_\delta$, yielding
$$
B(x) = \frac{1}{2\pi i} \int_{\mathcal{K}_\delta} \frac{x^s}{s} \frac{f(s)}{(s-1)^{1/2}}\, ds + O_\epsilon(x^{1/2+10\epsilon}).
$$
Letting $\delta \rightarrow 0$ shows this is
$$
= \frac{1}{\pi} \int_\sigma^1  \frac{x^s}{(1-s)^{1/2} s} f(s)\, ds + O_\epsilon(x^{1/2+10\epsilon}) = \frac{1}{\pi} \int_{1/2}^1 \frac{x^s}{(1-s)^{1/2} s} f(s)\, ds + O_\epsilon(x^{1/2+10\epsilon}),
$$
which yields the claim.
\end{proof}

Note that Radziejewski \cite{radziejewski2014oscillatory} has shown that this estimate is close to optimal in the sense that Theorem \ref{thm:B_approx} is not true if the error term is replaced by $O(x^{1/2}(\log x)^{-3/2-\epsilon})$ for any $\epsilon > 0$.

We note that by a similar contour shifting argument:

\begin{thm}
On the assumption of the Riemann Hypothesis, for $z > 0$,
$$
D_z(x) = \overline{D}_z(x) + O_\epsilon(x^{1/2+\epsilon}),
$$
for all $\epsilon > 0$, where $D_z(x)$ is defined in \eqref{eq:D_z_def}, and $\overline{D}_z(x)$ is defined in \eqref{eq:D_z_approx}.
\end{thm}

\section{Approximating $\mathbb{V}_b(X,p)$ for large $p$}
\label{sec:appendix large p}

We have noted that in Figure \ref{fig:data_and_prediction_AP}, for $\delta$ near $0$, there is a bump in data which does not appear in Conjecture \ref{b_variance_conj_q}. For reasons both numerical and theoretical we believe this bump fades away $X\rightarrow\infty$; in this appendix we explain now how this phenomena may be understood in terms of a lower order term, at least for $\delta$ sufficiently small in terms of $X$.

Connors and Keating \cite{connors1997} conjectured that 
\begin{equation}
\sum_{n \le x} b(n) b(n+q) \sim \frac{x}{2\log x} \alpha(q)
\end{equation}
for a precise positive multiplicative function $\alpha$. For odd primes $p$, the conjecture reads
\begin{equation}\label{eq:kc}
\sum_{n \le x} b(n) b(n+p) \sim \frac{x}{2\log x} \cdot \begin{cases} (1+\frac{1}{p}) & \mbox{if $p\equiv 3 \bmod 4$} \\ 1 & \mbox{if $p \equiv 1 \bmod 4$} \end{cases}.
\end{equation}
\begin{figure}[h]
	\includegraphics[width=\linewidth]{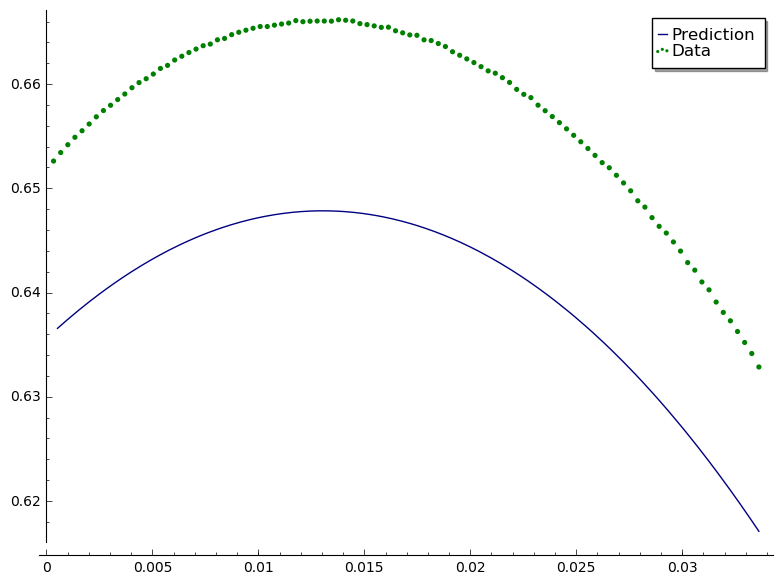}
	\caption{Numerically produced data compared to the Connors-Keating induced prediction: Consider the count of elements of $S$ less than $X = 9 \cdot 10^8$ congruent to a random $a$ modulo $p$, and let $\mathbb{V}_b(X,p)$ be the variance of these counts as $a$ varies.	Let $\delta = \delta_p = \log(X/p)/\log(X)$. For a selection of primes $p$ in between $X/2$ and $X$ -- for which $\delta \in (0,\log 2 / \log (X))$ -- we plot the points $(\delta, \mathbb{V}_b(X,p)/((X/p)/\sqrt{\log X}))$ under the label \textbf{data}, while the \textbf{prediction} is a plot of the curve $(\delta, K + (-K^2 X^{\delta} + 1-X^{-\delta})/\sqrt{\log X})$.}
\end{figure}

\begin{thm}
\label{thm:HL_correction}
Suppose that \eqref{eq:kc} holds with an error term of $O(x/(\log x)^{3/2})$, uniformly for $p=O(x)$. Fix $\epsilon>0$. Suppose $p = X^{1-\delta}$ is a prime with $\delta \in (\epsilon/\log X, \log 2/\log X)$. (In particular, $p \in (X/2,X)$.) Then, as $X \to \infty$,
\begin{equation}
\label{HL_prediction}
\frac{\mathbb{V}_b(X,p)}{(X/p)/\sqrt{\log X}} = K + \frac{-K^2 X^{\delta} + 1-X^{-\delta}}{\sqrt{\log X}}+ O_{\epsilon}\left(\frac{1}{\log X}\right).
\end{equation}
\end{thm}
Observe that the function $-K^2 X^{\delta} + 1-X^{-\delta}$ in increasing up to $\delta = \log(1/K)/\log X$, explaining the initial increase in Figure~\ref{fig:data_and_prediction_AP}.
\begin{proof}
Note that $\phi(p)=p-1$ and that there are $O(1)$ multiples of $p$ up to $X$. Hence  $\sum_{\substack{n\le X \\ (n,p)=1}} b(n)/\phi(p)  = B(X)/(p-1) + O(1/X)$. Moreover, in the arithmetic progression $n \equiv a \bmod p$ there are $2$ elements up to $X$ if $a \le X-p$, and a single element if $X-p<a\le X$. Simplifying $\mathbb{V}_b(X,p)$ using these observations yields
\begin{equation}
\mathbb{V}_b(X,p) = \frac{B(X)}{p-1} -\left(\frac{B(X)}{p-1} \right)^2 + \frac{2}{(p-1)}\sum_{a \le X-p} b(a)b(a+p) + O\left(\frac{1}{X}\right).
\end{equation}
The result now follows from \eqref{Landau_asymp} and our assumption on \eqref{eq:kc}.
\end{proof}

With more work one can make a prediction similar to this one for $\delta \leq  C/\log X$ for any constant $C$. Note that there is no inconsistently between Theorem \ref{thm:HL_correction} and Conjecture \ref{b_variance_conj_q}, as for $\delta = o(1)$ we have $K G(1-\delta) = K+o(1)$. Nonetheless the right hand side of \eqref{HL_prediction} plainly disagrees with $K G(1-\delta)$ if $\delta \neq o(1)$. It would be interesting to understand lower order terms for all $\delta \in (0,1)$, but we do not pursue this here.

\end{appendix}
\bibliographystyle{abbrv}
\bibliography{references}

\Addresses

\end{document}